\documentclass[11pt,reqno]{amsart}
\usepackage{enumerate}
\usepackage{amsrefs}
\usepackage{amsfonts, amsmath, amssymb, amscd, amsthm, bm, cancel}
\usepackage{mathrsfs}
\usepackage{url}
\usepackage{graphicx}
\usepackage[
linktocpage=true,colorlinks,citecolor=magenta,linkcolor=blue,urlcolor=magenta]{hyperref}
\usepackage{multicol}
\usepackage{comment}
\usepackage{amsrefs}
\usepackage[margin=1.3in]{geometry}
\newtheorem{prop}{Proposition}[section]
\newtheorem{thm}[prop]{Theorem}

\newtheorem{lem}[prop]{Lemma}

\newtheorem{defn}[prop]{Definition}

\newtheorem{assu}[prop]{Assumption}
\theoremstyle{definition}
\newtheorem*{ack}{Acknowledgments}
\theoremstyle{remark}
\newtheorem{rem}[prop]{Remark}

\numberwithin{equation}{section}
\allowdisplaybreaks

\begin{document}

\title[Volume preserving nonhomogeneous Gauss curvature flow]{Volume preserving nonhomogeneous Gauss curvature flow in hyperbolic space}

\author[Y. Wei]{Yong Wei}
\address{School of Mathematical Sciences, University of Science and Technology of China, Hefei 230026, P.R. China}
\email{\href{mailto:yongwei@ustc.edu.cn}{yongwei@ustc.edu.cn}}
\author[B. Yang]{Bo Yang}
\address{Department of Mathematical Sciences, Tsinghua University, Beijing 100084, P.R. China}
\email{\href{mailto:ybo@tsinghua.edu.cn}{ybo@tsinghua.edu.cn}}
\author[T. Zhou]{Tailong Zhou}
\address{School of Mathematics, Sichuan University, Chengdu 610065, Sichuan, P.R. China}
\email{\href{mailto:zhoutailong@scu.edu.cn}{zhoutailong@scu.edu.cn}}


\begin{abstract}
We consider the volume preserving flow of smooth, closed and convex hypersurfaces in the hyperbolic space $\mathbb{H}^{n+1}$ with speed given by a general nonhomogeneous function of the Gauss curvature. For a large class of speed functions, we prove that the solution of the flow remains convex, exists for all positive time $t\in [0,\infty)$ and converges to a geodesic sphere exponentially as $t\to\infty$ in the smooth topology. A key step is to show the $L^1$ oscillation decay of the Gauss curvature to its average along a subsequence of times going to the infinity, which combined with an argument using the hyperbolic curvature measure theory implies the Hausdorff convergence.
\end{abstract}

\maketitle



\section{Introduction}
Let $X_0:M^n\to\mathbb{H}^{n+1}$ be a smooth embedding such that $M_0=X_0(M)$ is a smooth, closed and convex hypersurface in the hyperbolic space $\mathbb{H}^{n+1}$. We consider the volume preserving curvature flow $X:M^n\times [0,T)\to \mathbb{H}^{n+1}$ satisfying
\begin{equation}\label{flow-VMCF}
	\left\{\begin{aligned}
		\frac{\partial}{\partial t}X(x,t)=&~\Big(\phi(t)-f(K(x,t))\Big)\nu(x,t),\\
		X(\cdot,0)=&~X_0(\cdot),
	\end{aligned}\right.
\end{equation}
where $\nu$ is the unit outward normal of $M_t=X(M,t)$, $K$ is the Gauss curvature of $M_t$ and
\begin{equation}\label{eqphi}
	\phi(t)=\frac{1}{|M_t|}\int_{M_t}f(K)\,\mathrm{d}\mu_t
\end{equation}
such that the domain $\Omega_t$ enclosed by $M_t$ has a fixed volume $|\Omega_t|=|\Omega_0|$ along the flow \eqref{flow-VMCF}. We assume that the function $f:[0,\infty)\to \mathbb{R}$ satisfies the following conditions:
\begin{assu}\label{ass} $f:[0,+\infty)\to\mathbb{R}$ is continuous and $C^2$ differentiable in $(0,+\infty)$, and satisfies the following conditions:
	\begin{enumerate}\setlength{\itemsep}{2pt}
		\item\label{p1} $f(x)>0$, $f'(x)>0$ for $x>0$;
		\item\label{p2} $\lim\limits_{x\to{+\infty}} f(x)=+\infty$;
		\item\label{p3} $\exists \Theta>0$ such that $f(x)\leq \Theta f'(x)x$ for $x>0$;
		\item\label{p4}  $xf''(x)+f'(x)\geq 0$ for $x>0$.
	\end{enumerate}
\end{assu}
\begin{rem}
	A particular example of functions satisfying Assumption \ref{ass} is
\begin{equation*}
  f(x)=x^{\alpha}, \quad \alpha>0.
\end{equation*}
In this case, the flow \eqref{flow-VMCF} is the classical volume preserving $\alpha$-Gauss curvature flow \eqref{s1.Kalpha} in $\mathbb{H}^{n+1}$, which was considered by the authors in \cite{WYZ2022}.  Other natural examples include the linear combination of powers
\begin{equation*}
  f(x)=\sum_{i=1}^{d}a_ix^{k_i},\quad a_i>0,\quad k_i>0
\end{equation*}
and
\begin{align*}
  f(x)=&\mathrm{e}^x-1,\\
 f(x)=&x^\alpha\ln(x+1),\quad \alpha>0.
\end{align*}
Moreover, any positive strictly increasing convex $C^2$ function $f$ with $f(0)=0$ satisfies Assumption \ref{ass}.
\end{rem}

Before describing our main result, we recall the following definitions of convexity of hypersurfaces in the hyperbolic space $\mathbb{H}^{n+1}$.
\begin{defn}$\ $
\begin{enumerate}
  \item A smooth closed hypersurface $M$ in $\mathbb{H}^{n+1}$ is called convex if all of its principal curvatures $\kappa_i, i=1,\cdots,n$ are positive everywhere on $M$.
      We also say that a hypersurface is weakly convex if all principal curvatures $\kappa_i\geq 0$.
  \item  A smooth closed hypersurface $M$ in $\mathbb{H}^{n+1}$ is called $h$-convex (also called horospherically convex), if all of its principal curvatures $\kappa_i\geq 1, i=1,\cdots,n$ everywhere on $M$.
  \item We say that a smooth closed hypersurface $M$ in $\mathbb{H}^{n+1}$ has positive sectional curvatures if its sectional curvatures $K(e_i,e_j)=\kappa_i\kappa_j-1>0$ for all $i\neq j\in  \{1,\dots,n\}$. This condition is weaker than $h$-convexity but is stronger than the convexity $\kappa_i>0,i=1,\cdots,n.$
\end{enumerate}
\end{defn}

Our main result of this paper is the following convergence result for the flow \eqref{flow-VMCF} with convex initial hypersurface.
\begin{thm}\label{theo}
Let $X_0:M^n\to \mathbb{H}^{n+1}$ be a smooth embedding such that $M_0=X_0(M)$ is a closed convex hypersurface in $\mathbb{H}^{n+1}$. Then the volume preserving flow \eqref{flow-VMCF} with $f$ satisfying Assumption \ref{ass} has a unique smooth convex solution $M_t$ for all time $t\in[0,\infty)$, and as $t\to \infty$ the solution $M_t$ converges smoothly and exponentially to a geodesic sphere of radius $\rho_{\infty}$ which encloses the same volume as $M_0$.
\end{thm}

The volume preserving mean curvature flow
\begin{equation}\label{eqH}
	\frac{\partial}{\partial t}X(x,t)=(\phi(t)-H)\nu(x,t)
\end{equation}
was introduced by Huisken \cite{Hui87} in 1987 for convex hypersurfaces in the Euclidean space $\mathbb{R}^{n+1}$, and it has been proved that for any smooth convex initial hypersurface, the flow \eqref{eqH} converges smoothly to a round sphere. There are further generalizations of the flow \eqref{eqH} for convex hypersurfaces in $\mathbb{R}^{n+1}$ with $H$ replaced by more general curvature functions including powers of the $k$th mean curvature $\sigma_k(\kappa),k=1,\cdots,n$. See \cite{AW21,BeSin18,BeSin18b,CS10,Mc04,Mc05,Mc17,Sine15} for instance. In particular, Bertini and Sinestrari \cite{BeSin18} considered the volume preserving non-homogeneous mean curvature flow of convex hypersurfaces in $\mathbb{R}^{n+1}$
\begin{equation}\label{s1.NH}
	\frac{\partial}{\partial t}X(x,t)=(\phi(t)-f(H))\nu(x,t),
\end{equation}
for a large class of nonhomogeneous function $f(H)$ of the mean curvature with $f$ satisfying some conditions.

The analogue of the flow \eqref{eqH} in the hyperbolic space $\mathbb{H}^{n+1}$ was first studied by Cabezas-Rivas and Miquel\cite{Cab-Miq2007} in 2007 assuming that the initial hypersurface is $h$-convex. The nonhomogeneous version \eqref{s1.NH} in $\mathbb{H}^{n+1}$ was considered by Bertini and Pipoli \cite{Be-Pip2016} also assuming $h$-convexity initially. Generalizations of the flow \eqref{eqH} in hyperbolic space with $H$ replaced by more general curvature functions were studied in \cite{BenChenWei,BenWei,GLW-CAG,Mak2012,WX}. In most cases, the $h$-convexity is assumed in order to prove the smooth convergence of the solution to geodesic spheres. The main reasons are that the $h$-convexity is convenient for the analysis of the curvature evolution equations and also $h$-convexity implies that the outer radius of the enclosed domain is uniformly controlled by its inner radius. In \cite{BenChenWei}, the first author with Andrews and Chen proved the smooth convergence of volume preserving $k$th mean curvature flows in $\mathbb{H}^{n+1}(n\geq 2)$ for initial hypersurfaces with positive sectional curvatures. This condition is weaker than $h$-convexity but still stronger than the convexity.

Recently, the authors \cite{WYZ2022} considered the volume preserving $\alpha$-Gauss curvature flow
\begin{equation}\label{s1.Kalpha}
	\frac{\partial}{\partial t}X(x,t)=(\phi(t)-K^\alpha)\nu(x,t)
\end{equation}
for convex hypersurfaces in the hyperbolic space $\mathbb{H}^{n+1}$. We showed that for any $\alpha>0$, the flow \eqref{s1.Kalpha} evolves any smooth, closed convex hypersurface in $\mathbb{H}^{n+1}$ to a geodesic sphere. The key ingredient we used in \cite{WYZ2022} is the projection method via the Klein model of the hyperbolic space which was described earlier by the first author and Andrews in \cite{BenWei}. Based on this we treated the flow as an equivalent flow in the Euclidean space and this allows us to derive a time-dependent positive lower bound on the principal curvatures, and a time-dependent upper bound on the Gauss curvature $K$. A continuity argument then implies that the flow exists for all positive time. Moreover, for the convergence result we used the curvature measure theory for convex bodies in $\mathbb{H}^{n+1}$ to show the Hausdorff convergence and an Alexandrov reflection argument to control the center of the inner ball of the evolving domains. The curve case of \eqref{s1.Kalpha} was also treated earlier by the first and second authors in \cite{WY2022}. The results in \cite{WY2022,WYZ2022} are the first results on non-local type volume preserving curvature flows in the hyperbolic space with only convexity required on the initial hypersurface.

Our Theorem \ref{theo} in this paper generalizes the results \cite{WY2022,WYZ2022} on volume preserving $\alpha$-Gauss curvature flow \eqref{s1.Kalpha} to volume preserving nonhomogeneous Gauss curvature flow \eqref{flow-VMCF} for a large class of nonhomogeneous functions. In this paper, instead of projecting the flow \eqref{flow-VMCF} to the Euclidean space as in \cite{WYZ2022}, we prove the curvature lower bound directly by analyzing an auxiliary function (see \S\ref{sec4} for details)
\begin{equation}\label{s1.Q}
	  	Q(p,t)=\log{\mathfrak{b}}+A\rho,
\end{equation}
where $\mathfrak{b}$ is the largest reciprocal of the principal curvatures $\kappa_i, i=1,\cdots,n$, and $\rho$ is the radial function of the evolving hypersurface $M_t$ on the time interval when $M_t$ is star-shaped with respect to some point $p_0$.  This avoids the use of more complicated projection method. To show the upper bound on the Gauss curvature, we employ Tso's method as in \cite{WYZ2022}. Since we only assumed convexity and the speed function $f$ is not necessarily homogeneous, the analysis of $f$ and the terms involving global term $\phi(t)$ need to be carefully treated.

For the long time existence of the flow, we need to show that the solutions remain smooth as long as the curvatures are bounded. For general non-homogeneous function $f$, the flow \eqref{flow-VMCF} may not be concave with respect to spacial second derivatives, we can not apply Krylov-Safonov's result to derive the $C^{2,\alpha}$ estimate of the solution. Instead, we apply a result in \cite[Theorem 6]{And2004} and view $f(K)$ as an increasing function of a concave operator $K^{1/n}$.  To study the asymptotical behavior of $M_t$ as $t\to\infty$, we note that the a prior estimates we obtained depend on the time and may degenerate as time $t\to\infty$. To overcome this difficulty, as in \cite{WYZ2022} we employ the curvature measure theory for convex bodies in hyperbolic space: For general nonhomogeneous function, we still have the monotonicity of the quermassintegral $\mathcal{A}_{n-1}(\Omega_t)$ of the evolving domain $\Omega_t$. This together with the long time existence of the flow implies that
\begin{equation}\label{s1.Kbar}
		\int_{M_{t_i}}{|K-\bar{K}|d\mu_{t_i}}\to 0,\quad \text{as}\,\, t_i\to\infty
	\end{equation}
along a subsequence of times $t_i\to\infty$, where $\bar{K}$ denotes the average integral of $K$. Then we can argue as in \cite{WYZ2022} to conclude the Hausdorff convergence using curvature measure theory and Alexandrov reflection method, and then improve it to the smooth convergence.

Throughout the proof, the presence of the nonhomogeneous speed function $f(K)$ introduces some technical difficulties. Assumption \ref{ass} is crucial for proving the a prior estimates of the flow \eqref{flow-VMCF}. In fact, the flow \eqref{flow-VMCF} is parabolic and has short-time existence due to item \eqref{p1}; items \eqref{p3} and \eqref{p4} are used to prove the time-dependent lower bound on the principal curvatures (see \S \ref{sec4}); and items \eqref{p2} - \eqref{p3} are used to derive the time-dependent upper bound on the Gauss curvature of the flow (see \S \ref{sec.upK}).
\begin{rem}$\ $
\begin{enumerate}
  \item We remark that nonhomogeneous curvature flows have been considered previously in the literature for the other type of curvature flows. See \cite{AleSin10,Guan23,LiLv20} for the contracting curvature flow case  and \cite{ChoT97,ChoT98,Pip22} for the expanding curvature flow case.
  \item The non-local type volume preserving curvature flows have also been studied in other aspects and in other ambient spaces. We refer the readers to \cite{AndWei20,Cab23} for surveys of the volume preserving type curvature flows and their geometric applications. Some other recent interesting works include  \cite{CSV20} for the volume preserving fractional mean curvature flow in the Euclidean space, \cite{Lam23} for singularity analysis of volume preserving mean curvature flow in the Euclidean space  and \cite{CaSche23} for a new non-local volume preserving mean curvature flow in the sphere.
\end{enumerate}
\end{rem}
The paper is organized as follows: In $\S$\ref{sec2}, we collect some preliminaries including the geometry of hypersufaces in the hyperbolic space, the evolution equations along the flow \eqref{flow-VMCF} and the quermassintegrals in hyperbolic space. In $\S$\ref{sec3}, we give the $C^0$ and $C^1$ estimates of the flow \eqref{flow-VMCF}. In $\S$\ref{sec4}, we prove the time-dependent positive lower bound for the principal curvatures along the flow \eqref{flow-VMCF}, by analyzing the auxiliary function \eqref{s1.Q}. In \S \ref{sec.upK} we prove a time-dependent upper bound on the Gauss curvature $K$.  This implies two-sided curvature bounds of the solution on any finite time interval, and then we obtain the long time existence of the flow \eqref{flow-VMCF} in $\S$\ref{sec5}. In \S \ref{sec.hau}, we show the subsequential Hausdorff convergence of $M_t$ and the convergence of the center of the inner ball of $\Omega_t$ to a fixed point. Finally, in \S \ref{Sec-main proof}, we complete the proof of Theorem \ref{theo}.

\begin{ack}
The research was supported by National Key R and D Program of China 2021YFA1001800 and 2020YFA0713100,  the Fundamental Research Funds for the Central Universities. The second author was also supported by China Postdoctoral Science Foundation No.2024M751605  and Shuimu Tsinghua Scholar Program (No. 2023SM102). The third author was also supported by NSFC (Grant No. 12401064). The authors would like thank Ruijia Zhang for discussions.
\end{ack}

\section{Preliminaries}\label{sec2}
In this section, we collect some preliminary results concerning the geometry of hypersurfaces in hyperbolic space,  the evolution equations for geometric quantities along the flow \eqref{flow-VMCF} and the quermassintegrals for bounded domains in hyperbolic space.
\subsection{Hyperbolic space}
The hyperbolic space $\mathbb{H}^{n+1}, n\geq 1$ can be viewed as a warped product manifold $(\mathbb{R}_{+}\times \mathbb{S}^n,g_{\mathbb{H}^{n+1}})$ with
\begin{equation*}
	g_{\mathbb{H}^{n+1}}=d\rho^2+\sinh^2\rho g_{\mathbb{S}^n},
\end{equation*}
where $g_{\mathbb{S}^n}$ is the round metric on unit sphere $\mathbb{S}^n$. Let $D$ be the Levi-Civita connection on $\mathbb{H}^{n+1}$. The vector field $V=\sinh \rho\partial_\rho$ is a conformal Killing field satisfying $DV=\cosh \rho g_{\mathbb{H}^{n+1}}.$


Let $\Omega$ be a convex domain in $\mathbb{H}^{n+1}$ with a smooth boundary $M=\partial\Omega$. Then $M$ is a smooth convex hypersurface in $\mathbb{H}^{n+1}$. We denote by $g_{ij}, h_{ij}$ and $\nu$ the induced metric, the second fundamental form and the unit outward normal vector of $M$ respectively. The eigenvalues of the Weingarten tensor $h_i^j:=g^{jk}h_{ki}$ are the principal curvatures $\kappa=(\kappa_1,\cdots,\kappa_n)$. As $M$ is convex, there exists a point $p_0\in\Omega$ such that $M$ is star-shaped with respect to $p_0$ and can be written as the radial graph $M=\{(\rho(\theta),\theta),~\theta\in \mathbb{S}^n\}$ with respect to $p_0$ for a smooth function $\rho\in C^\infty(\mathbb{S}^n)$. Equivalently, the support function of $M$ with respect to the point $p_0\in \Omega$ defined by
\begin{equation*}
  u=\langle V,\nu\rangle=\langle \sinh \rho\partial_\rho,\nu\rangle
\end{equation*}
is positive everywhere on $M$.  It is well known that (see e.g.\cite[\S 4]{GL15})
\begin{align}
	g_{ij}&=\rho_{i}\rho_{j}+\sinh^2\rho\sigma_{ij},\label{ex-g}\\
	h_{ij}&=\frac{1}{\sqrt{\sinh^2\rho+|\bar{\nabla} \rho|^2}}(-(\sinh \rho)\rho_{ij}+2(\cosh \rho)\rho_i \rho_j+\sinh^2\rho\cosh \rho\sigma_{ij}),\label{ex-h}\\
\nu&=\frac{1}{\sqrt{1+|\bar{\nabla}\rho|^2/\sinh^2\rho}}\left(1,-\frac{\rho_1}{\sinh^2\rho},\cdots,-\frac{\rho_n}{\sinh^2\rho}\right),\label{ex-nu}\\
      u&=\frac{\sinh^2\rho}{\sqrt{\sinh^2\rho+|\bar{\nabla}\rho|^2}},\label{ex-u}
\end{align}
where $\bar{\nabla}$ is the covariant derivative on $\mathbb{S}^n$ with respect to the round metric $g_{\mathbb{S}^n}=(\sigma_{ij})$ and $\rho_i=\bar{\nabla}\rho, \rho_{ij}=\bar{\nabla}_i\bar{\nabla}_j\rho$. It follows that the Gauss curvature of $M$ can be expressed as a function of $\rho$ and its up to second derivatives:
\begin{equation}\label{eq-Gauss}
	K=\frac{\det h_{ij}}{\det g_{ij}}=\frac{\det(-\sinh \rho\rho_{ij}+2\cosh \rho\rho_i \rho_j+\sinh^2\rho\cosh \rho\sigma_{ij})}{(\sinh^2\rho+|\bar{\nabla} \rho|^2)^{\frac{n+2}{2}}(\sinh \rho)^{2(n-1)}}.
\end{equation}

\subsection{Evolution equations}
Let $M_t$ be a smooth solution to the curvature flow \eqref{flow-VMCF} in the hyperbolic space $\mathbb{H}^{n+1}$. We have the following evolution equations (see \cite{BenWei}) for the induced metric $g_{ij}$, the area element $d{\mu_t}$ and the speed function $f(K)$:
\begin{align}
	\frac{\partial}{\partial t}g_{ij}&=2\left(\phi(t)-f(K)\right)h_{ij},\label{eq-g}\\
	\frac{\partial}{\partial t}d\mu_t&=H\left(\phi(t)-f(K)\right)d\mu_t,\label{eq-dmu}\\
	\frac{\partial}{\partial t}f(K)&=f'\dot{K}^{ij}\left(\nabla_i\nabla_j{f}+(f-\phi(t))(h_i^k h_k^j-\delta_i^j)\right),\label{eq-KK}
\end{align}
where $f'=f'(K)$ is the derivative of $f$, $\dot{K}^{ij}$ denote the derivatives with respect to components of the second fundamental form, and $\nabla$ denotes the Levi-Civita connection on $M_t$ with respect to the induced metric $g_{ij}$.

Let $M_t$ be a smooth convex solution to the flow \eqref{flow-VMCF} on the time interval $[0,T)$ and assume that $\{b_m^n\}$ is the inverse matrix of the Weingarten matrix $\{h_i^j\}$. The following lemma gives parabolic type evolution equations of $h_i^j$ and $b_m^n$.

\begin{lem}
Along the flow \eqref{flow-VMCF}, the Weingarten matrix $h_i^j$ of $M_t$ evolves by
\begin{align} \partial_t{h_i^j}-f'\dot{K}^{k\ell}\nabla_k\nabla_{\ell}{h_i^j}=&(f''\dot{K}^{k\ell}\dot{K}^{pq}+f'\ddot{K}^{k\ell,pq})\nabla_i{h_{k\ell}\nabla^j{h_{pq}}}\nonumber\\
&+f'(HK+\sigma_{n-1}(\kappa))h_i^j-nf'K(h_i^ph_p^j+\delta_i^j)\notag\\
		&+(f-\phi(t))(h_i^k h_k^j-\delta_i^j).\label{ev-Wm}
	\end{align}
Furthermore, if $M_t$ is a smooth convex solution to the flow \eqref{flow-VMCF} on the time interval $[0,T)$ and we assume that $\{b^{k\ell}\}$ is the inverse matrix of the second fundamental form $\{h_{ij}\}$ and $\{b_m^n:=g_{ms}b^{sn}\}$ is the inverse matrix of the Weingarten matrix $\{h_i^j\}$, then $b_m^n$ evolves by
\begin{align}
	\partial_t b_m^n-f'K b^{k\ell}\nabla_k\nabla_{\ell}{b_m^n}=&-f'K b^{k\ell}\left(b_i^n b_m^p b_q^s\nabla_k{h_s^i}\nabla_{\ell}{h_p^q}+b
	_i^n b_m^p b_q^s\nabla_k{h_p^q}\nabla_{\ell}{h_s^i}\right)\notag\\
	&-(f''K^2+f'K)b^{k\ell}b^{pq}b_m^i b_j^n\nabla_i{h_{k\ell}}\nabla^j{h_{pq}}\notag\\
&+f'Kb^{kp}b^{ql}b_m^i b_j^n\nabla_i{h_{k\ell}}\nabla^j{h_{pq}}\notag\\
	&-f'(HK+\sigma_{n-1}(\kappa))b_m^n
	+nf'K(b_m^i b_i^n+\delta_m^n)\notag\\
	&-(f-\phi(t))(\delta_m^n-b_m^i b_i^n).\label{ev-bmn}
\end{align}
\end{lem}
\begin{proof}
	By \cite[Lemma 2.4]{BenWei}, we have
 \begin{align}
    \partial_t{h_i^j}=&\dot{f}^{k\ell}\nabla_k\nabla_{\ell}{h_i^j}+\ddot{f}^{k\ell,pq}\nabla_i{h_{k\ell}\nabla^j{h_{pq}}}\nonumber\\
    &+(\dot{f}^{k\ell}h_k^rh_{r\ell}+\dot{f}^{k\ell}g_{k\ell})h_i^j-\dot{f}^{k\ell}h_{k\ell}(h_i^p h_p^j+\delta_i^j)\notag\\
    &+(f-\phi(t))(h_i^k h_k^j-\delta_i^j), \label{eq-h_i^j}
 \end{align}
where $\dot{f}^{k\ell}, \ddot{f}^{k\ell,pq}$ denote the first and second derivatives of $f$ with respect to the components of second fundamental form $(h_{ij})$. Since $f=f(K)$ is a function of Gauss curvature $K$, we have
\begin{align}
  \dot{f}^{k\ell}=&f'\dot{K}^{k\ell},\qquad \dot{f}^{k\ell}h_{k\ell}=nf'K,\label{eq-dotf1}\\
  \ddot{f}^{k\ell,pq}=&f'\ddot{K}^{k\ell,pq}+f''\dot{K}^{k\ell}\dot{K}^{pq}
\end{align}
and
\begin{align}
\dot{f}^{k\ell}h_k^rh_{r\ell}+\dot{f}^{k\ell}g_{k\ell}=&f'(\dot{K}^{k\ell}h_k^rh_{r\ell}+\dot{K}^{k\ell}g_{k\ell})\nonumber\\
=&f'(HK+\sigma_{n-1}(\kappa)).\label{eq-dotf2}
\end{align}
Substituting \eqref{eq-dotf1} - \eqref{eq-dotf2} into \eqref{eq-h_i^j} gives \eqref{ev-Wm}.

To derive equation \eqref{ev-bmn}, we calculate using $b^n_mh_n^k=\delta_m^k$ that
\begin{align}
	\partial_t{b_m^n}=&-b_m^i b_j^n\partial_t{h_i^j}\notag\\
	=&-f'\dot{K}^{k\ell}b_m^i b_j^n\nabla_k\nabla_{\ell}{h_i^j}\nonumber\\
&-(f''\dot{K}^{k\ell}\dot{K}^{pq}+f'\ddot{K}^{k\ell,pq})b_m^i b_j^n\nabla_i{h_{k\ell}}\nabla^j{h_{pq}}\notag\\
	&-f'(HK+\sigma_{n-1})b_m^n+nf'K(b_m^i b_i^n+\delta_m^n)\nonumber\\
&-(f-\phi(t))(\delta_m^n-b_m^i b_i^n).\label{eq-b}
\end{align}
Note that
\begin{align}
	\nabla_{\ell}{b_m^n}&=-b_i^n b_m^s\nabla_{\ell}{h_s^i},\label{nablabh}\\
	\nabla_k\nabla_{\ell}b_m^n&=-b_m^i b_j^n \nabla_k\nabla_{\ell}{h_i^j}+b_i^n b_m^p b_q^s\nabla_k{h_p^q}\nabla_{\ell}{h_s^i}\nonumber\\
&\qquad +b_i^n b_m^p b_q^s\nabla_k{h_s^i}\nabla_{\ell}{h_p^q},\label{nnbh}\\
	\dot{K}^{k\ell}&=Kb^{k\ell},\label{nabla-K}\\
	\ddot{K}^{k\ell,pq}&=Kb^{k\ell}b^{pq}-Kb^{kp}b^{q\ell}.\label{K-kl,pq}
\end{align}
Combining \eqref{eq-b}-\eqref{K-kl,pq} gives \eqref{ev-bmn}.
\end{proof}

On the time interval when $M_t$ is star-shaped with respect to some point $p_0$, the radial function $\rho$ evolves by 
\begin{equation}\label{ev-rho}
	\frac{\partial}{\partial t}\rho=\left(\phi(t)-f(K)\right)\sqrt{1+\frac{|\bar{\nabla}\rho|^2}{\sinh^2{\rho}}},
\end{equation}
where $K$ is expressed in \eqref{eq-Gauss} as a function of $\bar{\nabla}^2\rho, \bar{\nabla}\rho$ and $\rho$. The support function $u(x,t)=\langle{\sinh \rho_{p_0}(x)\partial_{\rho_{p_0}},\nu}\rangle$ of $M_t$ with respect to $p_0$  evolves by (see Lemma 4.3 in \cite{BenWei})
\begin{equation}\label{equeq}
	\frac{\partial}{\partial t}{u}=f'\dot{K}^{ij}\nabla_i\nabla_j{u}+\cosh{\rho_{p_0}}(x)\Big(\phi(t)-f-nKf'\Big)+f'KHu.
\end{equation}

The following lemma will be used to prove the lower bound of the principal curvatures.
\begin{lem}
  We have
  \begin{equation}\label{s2.2drho}
    \nabla_i\rho=\langle \partial_\rho,e_i\rangle,\qquad \nabla_j\nabla_i\rho=\coth\rho\left(g_{ij}-\nabla_i\rho\nabla_j\rho\right)-\frac{uh_{ij}}{\sinh\rho}.
  \end{equation}
\end{lem}
\proof
Since $V=\sinh\rho\partial_\rho$ is a conformal Killing vector field, i.e.,
 \begin{equation}\label{s2.conf}
   \langle D_X(\sinh \rho\partial_\rho), Y\rangle=\cosh \rho \langle  X,Y\rangle
 \end{equation}
for any tangential vector fields $X,Y$ in $\mathbb{H}^{n+1}$, we have   (see, e.g., \cite[\S 2]{GL15})
\begin{align*}
  \nabla_i\cosh\rho= & \langle\sinh\rho\partial_\rho,e_i\rangle,  \\
  \nabla_j  \nabla_i\cosh\rho=& \cosh\rho g_{ij}-uh_{ij}.
\end{align*}
Observing that
\begin{align*}
  \nabla_i\cosh\rho=&\sinh\rho\nabla_i\rho,\\
  \nabla_j  \nabla_i\cosh\rho=&\sinh\rho\nabla_j\nabla_i\rho+\cosh\rho\nabla_j\rho\nabla_i\rho,
\end{align*}
we conclude the equation \eqref{s2.2drho}.
\endproof

\subsection{Quermassintegrals}
Let $\mathcal{K}(\mathbb{H}^{n+1})$ be the set of compact convex sets in $\mathbb{H}^{n+1}$ with nonempty interior.  For any $\Omega\in \mathcal{K}(\mathbb{H}^{n+1})$ , the quermassintegrals of $\Omega$ are defined as follows (see \cite[Definition 2.1]{Sol05} \footnote{Note that the definition for $\mathcal{A}_k$ given here is the same as that for $W_{k+1}$ given in \cite{Sol05} up to a constant. In fact, we have $\mathcal{A}_k=(n+1)\binom{n}{k}W_{k+1}$.}):
\begin{equation}\label{Wk}
	\mathcal{A}_k(\Omega)=(n-k)\binom{n}{k}\frac{\omega_k\cdots\omega_0}{\omega_{n-1}\cdots\omega_{n-k-1}}\int_{\mathcal{L}_{k+1}}{\chi(L_{k+1}\cap\Omega)dL_{k+1}}
\end{equation}
for $k=0,1,\dots,n-1$, where $\omega_k=|\mathbb{S}^k|$ denotes the area of $k$-dimensional unit sphere,  $\mathcal{L}_{k+1}$ is the space of $(k+1)$-dimensional totally geodesic subspaces $L_{k+1}$ in $\mathbb{H}^{n+1}$ and $\binom{n}{k}=\frac{n!}{k!(n-k)!}$. The function $\chi$ is defined to be 1 if $L_{k+1}\cap\Omega\neq\emptyset$ and to be 0 otherwise. In particular, $\mathcal{A}_0(\Omega)=|\partial\Omega|$. Furthermore, we set
\begin{equation*}
\mathcal{A}_{-1}(\Omega)=|\Omega|,\qquad \mathcal{A}_{n}(\Omega)=\frac{\omega_{n}}{n+1}.
\end{equation*}

If the boundary $M=\partial\Omega$ is smooth (or at least of class $C^2$), we can define the principal curvatures $\kappa=(\kappa_1,\dots,\kappa_n)$ as the eigenvalues of the Weingarten matrix $\mathcal{W}$ of $M$. For each $k\in\{1,\dots,n\}$, the $k$th mean curvature $\sigma_k$ of $M$ is then defined as the $k$th elementary symmetric function of the principal curvatures of $M$:
\begin{equation*}
	\sigma_k=\sum_{1\leq i_1<\cdots<i_k\leq n}{\kappa_{i_1}\cdots\kappa_{i_k}}.
\end{equation*}
These include the mean curvature $H=\sigma_1$ and Gauss curvature $K=\sigma_n$ as special cases.
In the smooth case, the quermassintegrals and the curvature integrals of a smooth convex domain $\Omega$ in $\mathbb{H}^{n+1}$ are related as follows:
\begin{align}%
	\mathcal{A}_1(\Omega)=&\int_{\partial\Omega}{\sigma_1}d\mu-n\mathcal{A}_{-1}(\Omega),\label{eq-V1}\\
	\mathcal{A}_k(\Omega)=&\int_{\partial\Omega}{\sigma_k}d\mu-\frac{n-k+1}{k-1}\mathcal{A}_{k-2}(\Omega),\quad k=2,\cdots,n \label{eq-VW}.
\end{align}
The quermassintegrals for smooth domains satisfy a nice variational property (see \cite{BA97}):
\begin{equation}\label{eqWk}
  \frac{d}{dt}\mathcal{A}_k(\Omega_t)=(k+1)\int_{M_t}\eta \sigma_{k+1}d\mu_t,\quad k=0,\cdots,n-1
\end{equation}
along any normal variation with velocity $\eta$.

The quermassintegrals defined by \eqref{Wk} are monotone with respect to inclusion of convex sets. That is, if $E,F\in \mathcal{K}(\mathbb{H}^{n+1})$ satisfy $E\subset F$, we have
\begin{equation}\label{s2.Akmo}
  \mathcal{A}_k(E)\leq \mathcal{A}_k(F)
\end{equation}
for all $k=0,1,\cdots,n$. Moreover, they are continuous with respect to the Hausdorff distance. Recall that the Hausdorff distance between two convex sets $\Omega, L\in \mathcal{K}(\mathbb{H}^{n+1})$ is defined as
\begin{equation*}
	\mathrm{d}_{\mathcal{H}}(\Omega,L):=\mathrm{inf}\{\lambda>0:\Omega\subset B_{\lambda}(L)\,  \text{and}\, L\subset B_{\lambda}(\Omega)\},
\end{equation*}
where $	B_{\lambda}(L):=\{q\in \mathbb{H}^{n+1}|~\mathrm{d}_{\mathbb{H}^{n+1}}(q,L)<\lambda\}$.

\section{A priori estimates}
In this section, we first show the $C^0$ and $C^1$ estimates of the solution $M_t$ to the flow \eqref{flow-VMCF}, which can be proved using a similar argument as in \cite{BenChenWei,BenWei}.  Then we derive a positive lower bound on the principal curvatures $\kappa_i$ of the solution $M_t$, and prove the upper bound of the Gauss curvature $K$.

\subsection{$C^0$ and $C^1$ estimates}\label{sec3}
Let $M_t$ be a smooth convex solution to the flow \eqref{flow-VMCF} on a maximal existence time interval $[0,T)$. Denote $\Omega_t$ the domain enclosed by $M_t$. As the velocity of the flow \eqref{flow-VMCF} only depends on the curvature and is invariant under reflection with respect to a totally geodesic hyperplane, we can argue as in \cite[\S 4]{BenChenWei} using the Alexandrov reflection method to show that the inner radius and outer radius of $\Omega_t$ are uniformly bounded:
\begin{lem}\label{Lemma-ioradius}
	Let $M_t$ be a smooth convex solution to the flow \eqref{flow-VMCF} on the time interval $t\in [0,T)$. Denote $\rho_-(t)$, $\rho_+(t)$ be the inner radius and outer radius of $\Omega_t$. Then there exist positive constants $c_1,\ c_2$ depending only on $n$ and $M_0$ such that
	\begin{equation}\label{in-out-radius}
	0<c_1\leq\rho_-(t)\leq\rho_+(t)\leq c_2
	\end{equation}
	for all time $t\in[0,T)$.
\end{lem}

By \eqref{in-out-radius}, the inner radius of $\Omega_t$ is bounded below by a positive constant $c_1$. This implies that there exists a geodesic ball of radius $c_1$ contained in $\Omega_t$ for each $t\in[0,T)$. The following lemma says that there exists a geodesic ball with fixed center enclosed by the flow hypersurface on a suitable fixed time interval. For the proof, we adapt a similar argument as in \cite[Lemma 4.2]{BenWei}. But for our nonhomogeneous case, we need the positivity and monotonicity of $f$ as in Assumption \ref{ass}.
\begin{lem}\label{lem3.2}
	Let $M_t$ be a smooth convex solution to the flow \eqref{flow-VMCF} on the time interval $[0,T)$. For any $t_0\in[0,T)$, let $B_{\rho_0}(p_0)$ be the inball of $\Omega_{t_0}$, where $\rho_0=\rho_-(t_0)$. Then
	\begin{equation}\label{inball-t0}
	B_{\rho_0/2}(p_0)\subset\Omega_t,\ \ t\in[t_0,\min\{T,t_0+\tau\})
	\end{equation}
	for some $\tau$ depending only on $n$ and $M_0$.
\end{lem}
\proof
Given $p_0$, we denote by $\rho_{p_0}$ the distance function to $p_0$ in $\mathbb{H}^{n+1}$ and by $\partial_\rho=\partial_{\rho_{p_0}}$ the gradient vector of $\rho_{p_0}$. For any $x\in M_t$,
\begin{align}\label{s3.inball-1}
  \frac{\partial}{\partial t} \sinh^2\rho_{p_0}(x)=& 2\langle \sinh \rho_{p_0}(x)\partial_\rho, \frac{\partial}{\partial t}(\sinh \rho_{p_0}(x)\partial_\rho) \rangle\nonumber\\
  = &2 \sinh \rho_{p_0}(x)\cosh \rho_{p_0}(x)\Big(\phi(t)-f(K(x,t))\Big)\langle \partial_\rho,\nu\rangle,
\end{align}
where we used the conformal property \eqref{s2.conf} of  $\sinh \rho\partial \rho$. It follows from \eqref{s3.inball-1} that
 \begin{align*}
  \frac{\partial}{\partial t} \rho_{p_0}(x)=& \Big(\phi(t)-f(K(x,t))\Big)\langle \partial \rho,\nu\rangle\\
  \geq&~-f(K(x,t))\langle \partial_\rho,\nu\rangle,
\end{align*}
 since $\phi(t)>0$ and $\langle \partial_\rho,\nu\rangle>0$ on $M_t$.

 Denote $\rho(t)=\min_{M_t}\rho_{p_0}(x)$. At the minimum point, we have $\langle \partial_\rho,\nu\rangle=1$ and $\kappa_i\leq \coth \rho(t)$. Since $f$ is strictly increasing, we have
 \begin{equation*}
   f(K)\leq f(\coth^n\rho(t))
 \end{equation*}
 at the minimum point and so
\begin{equation}\label{s3.inball-3}
  \frac d{dt}\rho(t)\geq -f(\coth^n\rho(t)).
\end{equation}

Let $\bar{\rho}(t)$ be the solution of the ODE
\begin{equation}\label{s3.ODE}
  \left\{\begin{aligned}
  \frac{d}{dt}\bar{\rho}(t)=&-f(\coth^n\bar{\rho}(t)),\\
  \bar{\rho}(t_0)=&\rho_0.
  \end{aligned}\right.
\end{equation}
Denote $\tau$ as the time such that $\bar{\rho}(t_0+\tau)=\rho_0/2$. Since $f$ is positive,  $\bar{\rho}(t)$ is strictly decreasing. It follows that the inverse $t=t(\bar{\rho})$ of the function $\bar{\rho}(t)$ is well-defined and satisfies
\begin{equation*}
  \frac{dt}{d\bar{\rho}}=-\frac{1}{f(\coth^n\bar{\rho})}.
\end{equation*}
Integrating this equation gives
\begin{align*}
  \tau= & \int_{t_0}^{t_0+\tau}dt \\
  =& -\int_{\rho_0}^{\rho_0/2}\frac{d\bar{\rho}}{f(\coth^n\bar{\rho})}\\
  =&\int_{\rho_0/2}^{\rho_0}\frac{ds}{f(\coth^ns)},
\end{align*}
which depends only on the bounds of $\rho_0$  and not on $t_0$.

Since $\rho(t_0)=\rho_0$, by comparing \eqref{s3.inball-3} and the ODE \eqref{s3.ODE}, we conclude that
\begin{equation*}
  \rho(t)\geq \frac{\rho_0}{2},\quad \forall~t\in [t_0,\min\{t_0+\tau,T\}).
\end{equation*}
This means that $B_{\rho_0/2}(p_0)\subset \Omega_t$ for all $t\in [t_0,\min\{t_0+\tau,T\})$.
\endproof

Let $M_t$ be a smooth convex solution to the flow \eqref{flow-VMCF} on the time interval $[0,T)$.  For any $t_0\in[0,T)$, let $B_{\rho_0}(p_0)$ be the inball of $\Omega_{t_0}$, where $\rho_0=\rho_-(t_0)$. Consider the support function $u(x,t)=\sinh \rho_{p_0}(x)\langle\partial \rho_{p_0},\nu\rangle$ of $M_t$ with respect to the point $p_0$, where $\rho_{p_0}$ is the distance function in $\mathbb{H}^{n+1}$ from the point $p_0$. Since $M_t$ is convex, by \eqref{in-out-radius} and \eqref{inball-t0}, we see
\begin{equation}\label{u-bound}
	\begin{split}
u(x,t)&\geq\sinh\left(\frac{\rho_0}{2}\right)\geq\sinh\left(\frac{c_1}{2}\right)=:2c,\\
u(x,t)&\leq\sinh(2c_2)
    \end{split}
\end{equation}
and
\begin{equation}\label{rhobound}
0<\frac{c_1}{2}\leq \rho_{p_0}(t)\leq 2c_2<\infty
\end{equation}
for any $t\in[t_0,\min\{T,t_0+\tau\})$. By \eqref{ex-u}, we have the $C^1$ estimate on $\rho$,
\begin{equation}\label{s3.C1}
  |\bar{\nabla}\rho|\leq \sqrt{\sinh^2\rho+|\bar{\nabla}\rho|^2}=\frac{\sinh^2\rho}{u}\leq\frac{\sinh^2(2c_2)}{\sinh\left({c_1}/{2}\right)}
\end{equation}
on the time interval $t\in[t_0,\min\{T,t_0+\tau\})$, where $|\bar{\nabla}\rho|$ is the norm of the gradient of $\rho$ with respect to the round metric on $\mathbb{S}^n$.

Moreover, we have
\begin{align}\label{s3.nablarho}
  1- |\nabla\rho|^2= & |D\rho|^2-|\nabla\rho|^2\nonumber\\
  =&\langle D\rho,\nu\rangle^2=\frac{u^2}{\sinh^2\rho}\nonumber\\
  \geq&\frac{\sinh^2(c_1/2)}{\sinh^2(2c_2)}~=:c_3
\end{align}
holds on $t\in[t_0,\min\{T,t_0+\tau\})$, where $D$ and $\nabla$ denote the Levi-Civita connection on $\mathbb{H}^{n+1}$ and on $M_t$ with respect to the induced metric. The estimate \eqref{s3.nablarho} will be used crucially in the proof of Proposition \ref{preserve convex}.

\subsection{Preserving convexity}\label{sec4}
In this subsection, we show that the solution $M_t$ of the flow \eqref{flow-VMCF} preserves the convexity. This follows from the following (time-dependent) lower bound on the principal curvatures of $M_t$.
\begin{prop}\label{preserve convex}
Let $M_0$ be a smooth, closed and convex hypersurface in $\mathbb{H}^{n+1}$, and $M_t,t\in[0,T)$ be the smooth solution of the flow \eqref{flow-VMCF} starting from $M_0$. If $T<\infty$, then there exist constants $\Lambda_1$ and $\Lambda_2$ depending only on $n, M_0$ and $\Theta$, such that the principal curvatures $\kappa_i$ of $M_t$ satisfy
	\begin{equation}\label{s4.2-0}
		\kappa_i\geq \Lambda_2^{-1}\mathrm{e}^{-\frac{2\Lambda_1}{\tau}t}
	\end{equation}
	for all $i=1,\dots,n$ and $t\in [0,T)$, where $\tau$ is the constant in Lemma \ref{lem3.2} and $\Theta$ is the constant in item \eqref{p3} of Assumption \ref{ass}.
\end{prop}
\begin{proof}
Since $M_0$ is convex, by continuity the solution $M_t$ is convex for at least a short time. Let $T_1<T$ be the largest time such that $M_t$ is convex for all $0\leq t<T_1$. If we can derive the estimate \eqref{s4.2-0} on the interval $[0,T_1)$, then it implies a contradiction with the maximality of $T_1$ and so that $M_t$ is convex on the whole time interval $[0,T)$, and the estimate \eqref{s4.2-0} holds for all $t\in [0,T)$. Therefore, without loss of generality, we can assume that $M_t$ is convex for $t\in [0,T)$ and  we need to show the estimate \eqref{s4.2-0} holds for all $t\in [0,T)$.
	
For any time $t_0\in[0,T)$, let $B_{\rho_0}(p_0)$ be the inball of $\Omega_{t_0}$, where $\rho_0=\rho_-(t_0)$. Denote
\begin{equation*}
  \mathfrak{b}_i:=\frac{1}{\kappa_i},\quad \text{and}\quad  \mathfrak{b}:=\max_{i=1,\cdots,n}{\mathfrak{b}_i}.
\end{equation*}
In order to prove the lower bound of the principal curvatures $\kappa_i$, it suffices to prove the upper bound of $\mathfrak{b}$.
	
	We consider the auxiliary function
	  \begin{equation}\label{eq-Q}
	  	Q(p,t):=\log{\mathfrak{b}}+A\rho,\qquad t\in [t_0,\min\{T,t_0+\tau\}),
	  \end{equation}
  where $A>0$ is a large constant to be determined and $\tau$ is the constant in Lemma \ref{lem3.2}. Suppose that the maximum of $Q$ on $M\times [t_0,\min\{T,t_0+\tau\})$ is attained at $(\bar{p},\bar{t})$. We choose a local orthonormal frame $e_1,\dots,e_n$ around $\bar{p}$ such that at $(\bar{p},\bar{t})$ we have
  \begin{equation*}
    g_{ij}=\delta_{ij},\qquad h_{ij}=\kappa_i\delta_{ij}.
  \end{equation*}
  By a rotation, we also assume that
  \begin{equation*}
    \mathfrak{b}(\bar{p},\bar{t})=\mathfrak{b}_1(\bar{p},\bar{t})=b^1_1(\bar{p},\bar{t}).
  \end{equation*}
  Let $\xi=(1,0\cdots,0)$ be a contravariant vector field and set
  \begin{equation*}
    \lambda=\dfrac{b^{ij}\xi_i\xi_j}{g^{ij}\xi_i\xi_j},
  \end{equation*}
  which is well-defined in a neighbourhood of $(\bar{p},\bar{t})$. Defining $\tilde{Q}(p,t)$ by replacing $\mathfrak{b}$ by $\lambda$ in \eqref{eq-Q}, we see that $\tilde{Q}$ attends its maximum at $(\bar{p},\bar{t})$. Moreover, at $(\bar{p},\bar{t})$ we have $\partial_t\lambda=\partial_tb_1^1$ and the spatial derivatives also coincide. That is, $\lambda$ satisfies the same evolution equation as $b_1^1$ at the point $(\bar{p},\bar{t})$. Therefore, for the sake of clarity (see a similar argument as in \cite[\S 4]{Ger14}), we can treat $b_1^1$ as a scalar function and pretend that $Q$ is defined by
  \begin{equation}\label{eq-Q2}
	  	Q(p,t)=\log{b_1^1}+A\rho,\qquad t\in [t_0,\min\{T,t_0+\tau\}).
	  \end{equation}

  As $\{b_m^n\}$ is diagonal at $(\bar{p},\bar{t})$ and $b^1_1(\bar{p},\bar{t})=\mathfrak{b}(\bar{p},\bar{t})$, by \eqref{ev-bmn} at the point $(\bar{p},\bar{t})$, we have
  \begin{align}
  	&\partial_t b_1^1-f'K^{k\ell}\nabla_k\nabla_{\ell}{b_1^1}\notag\\
  	=&-\mathfrak{b}^2(f''K^2+f'K)(b^{kk}\nabla_1{h_{kk}})^2-\mathfrak{b}^2f'Kb^{kk}b^{\ell\ell}(\nabla_1{h_{k\ell}})^2\notag\\
  	&-\mathfrak{b}f'(HK+\sigma_{n-1}(\kappa))+nf'K(\mathfrak{b}^2+1)-(f-\phi(t))(1-\mathfrak{b}^2).\label{ev-mathfrakb}
  \end{align}
  Combining \eqref{ev-rho} and \eqref{s2.2drho}, we also have
  \begin{align}\label{s3.evl-rho}
    \partial_t\rho-f'\dot{K}^{k\ell}\nabla_k\nabla_\ell\rho= & (\phi(t)-f(K))\sqrt{1+\frac{|\bar{\nabla}\rho|^2}{\sinh^2{\rho}}}+nf'K\frac{u}{\sinh\rho} \nonumber\\
     & -f'\coth{\rho}\left(\sigma_{n-1}(\kappa)-K b^{k\ell}\nabla_k\rho\nabla_{\ell}\rho\right).
  \end{align}

If $\bar{t}=t_0$, we have
\begin{equation}\label{Q-initial}
	Q(p,t)\leq Q(\bar{p},t_0)\leq \log\max_{p\in M}{\mathfrak{b}(p,t_0)}+2Ac_2
\end{equation}
for $(p,t)\in M\times [t_0,\min\{T,t_0+\tau\})$. In the following, we assume $\bar{t}>t_0$.  We shall apply the maximum principle to the evolution equation of $Q$.

Since $(\bar{p},\bar{t})$ is a maximum point of $Q$, at $(\bar{p},\bar{t})$ there hold
\begin{equation}\label{nablaQ=0}
	0=\nabla_i Q=\frac{\nabla_ib_{11}}{\mathfrak{b}}+A\nabla_i\rho
\end{equation}
and
\begin{align}
	0\leq& \partial_t Q-f'\dot{K}^{k\ell}\nabla_k\nabla_{\ell}{Q}\notag\\
	=&\frac{1}{\mathfrak{b}}\left(\partial_t b_{1}^1-f'\dot{K}^{k\ell}\nabla_k\nabla_{\ell}b_1^1\right)+\frac{f'}{\mathfrak{b}^2}\dot{K}^{k\ell}\nabla_k b_{11}\nabla_{\ell}b_{11}\nonumber\\
&+A\left(\partial_t \rho-f'\dot{K}^{k\ell}\nabla_k\nabla_{\ell}{\rho}\right)\notag\\
	=&-\mathfrak{b}(f''K^2+f'K)(b^{kk}\nabla_1{h_{kk}})^2-\mathfrak{b}f'Kb^{kk}b^{\ell\ell}(\nabla_1{h_{k\ell}})^2\notag\\
	&-f'(HK+\sigma_{n-1}(\kappa))+nf'K(\mathfrak{b}+\frac{1}{\mathfrak{b}})-(f-\phi(t))(\frac{1}{\mathfrak{b}}-\mathfrak{b})\notag\\
	&+\frac{f'K}{\mathfrak{b}^2} b^{kk}(\nabla_k b_{11})^2+A(\phi(t)-f)\sqrt{1+\frac{|\bar{\nabla}\rho|^2}{\sinh^2{\rho}}}\notag\\
	&-Af'\coth{\rho}\left(\sigma_{n-1}(\kappa)-K  b^{kk}(\nabla_k\rho)^2\right)+nAf'K\frac{u}{\sinh\rho}\notag\\
	=&:Q_1+Q_2,\label{eq-sepa}
\end{align}
where $Q_1$ denote the terms involving $\phi(t)$:
\begin{equation}\label{eq-Q_1}
	Q_1=\phi(t)\left(\frac{1}{\mathfrak{b}}-\mathfrak{b}+A\sqrt{1+\frac{|\bar{\nabla}\rho|^2}{\sinh^2{\rho}}}\right),
\end{equation}
and $Q_2$ denote the remaining terms:
\begin{align}
	Q_2=&-\mathfrak{b}(f''K^2+f'K)(b^{kk}\nabla_1{h_{kk}})^2\nonumber\\
&-\mathfrak{b}f'Kb^{kk}b^{\ell\ell}(\nabla_1{h_{k\ell}})^2+\frac{f'K}{\mathfrak{b}^2} b^{kk}(\nabla_k b_{11})^2\notag\\
	&-f'(HK+\sigma_{n-1}(\kappa))+nf'K(\mathfrak{b}+\frac{1}{\mathfrak{b}})-f(\frac{1}{\mathfrak{b}}-\mathfrak{b})\notag\\
	&-Af'\coth{\rho}\left(\sigma_{n-1}(\kappa)-K  b^{kk}(\nabla_k\rho)^2\right)\notag\\
	&-Af\sqrt{1+\frac{|\bar{\nabla}\rho|^2}{\sinh^2{\rho}}}+nAf'K\frac{u}{\sinh\rho}.\label{eq-Q_2}
\end{align}

\textbf{Estimate of $Q_1$:} By \eqref{ex-u} and \eqref{u-bound}, \eqref{rhobound}, we see that
\begin{equation*}
  \sqrt{1+\frac{|\bar{\nabla}\rho|^2}{\sinh^2{\rho}}}=\frac{\sinh\rho}{u}\leq \frac{\sinh(2c_2)}{\sinh(c_1/2)}~=:c_4.
\end{equation*}
Then, if
\begin{equation}\label{con-b1}
\mathfrak{b}>\frac{Ac_4+\sqrt{A^2c_4^2+4}}{2},
\end{equation}
we have $Q_1<0$.

\textbf{Estimate of $Q_2$:} Firstly, by item \eqref{p4} in Assumption \ref{ass}, we have
\begin{align}
	&-\mathfrak{b}(f''K^2+f'K)(b^{kk}\nabla_1{h_{kk}})^2\leq ~0. \label{es-Q_2-0}
\end{align}
Using \eqref{nablabh} and \eqref{nabla-K}, we also have
\begin{align}
	&-\mathfrak{b}f'Kb^{kk}b^{\ell\ell}(\nabla_1{h_{k\ell}})^2+\frac{f'K}{\mathfrak{b}^2}b^{kk}(\nabla_k b_{11})^2\notag\\
	=&-\mathfrak{b}f'Kb^{kk}b^{\ell\ell}(\nabla_1{h_{k\ell}})^2+f'\mathfrak{b}^2 Kb^{kk}(\nabla_k h_{11})^2\notag\\
	\leq&0,\label{es-Q_2-2}
\end{align}
where the inequality is obtained by discarding the terms with $\ell\neq 1$. For the fourth line of \eqref{eq-Q_2},
\begin{align}
	\text{Line 4 of }\eqref{eq-Q_2} \leq & -Af'\coth{\rho}\left(\sigma_{n-1}(\kappa)-K  \left(\sum_k{b^{kk}}\right)|\nabla\rho|^2\right)\nonumber\\
 =& -Af'\coth{\rho}~\sigma_{n-1}(\kappa)\left(1-|\nabla\rho|^2\right)\nonumber\\
 \leq & -c_3Af'\coth{\rho}~\sigma_{n-1}(\kappa)\nonumber\\
 \leq & -c_3Af'\sigma_{n-1}(\kappa),\label{es-Q_2-4}
\end{align}
where we used the estimate \eqref{s3.nablarho} and that $c_3$ is the constant in \eqref{s3.nablarho} which depends only on $n$ and $M_0$.

Substituting \eqref{es-Q_2-0}-\eqref{es-Q_2-4} into \eqref{eq-Q_2}, we have
\begin{align}
	Q_2\leq& -f'(HK+\sigma_{n-1}(\kappa))+nf'K(\mathfrak{b}+\frac{1}{\mathfrak{b}})-f(\frac{1}{\mathfrak{b}}-\mathfrak{b})\notag\\
	&-c_3Af'~\sigma_{n-1}(\kappa)-Af\sqrt{1+\frac{|\bar{\nabla}\rho|^2}{\sinh^2{\rho}}}+nAf'K\frac{u}{\sinh\rho}\notag\\
\leq & -f'\sigma_{n-1}(\kappa)\left(1+c_3A\right)+(nf'K+f)\mathfrak{b}\nonumber\\
&-f'HK+\frac{nf'K}{\mathfrak{b}}+nAf'K\frac{u}{\sinh\rho}\nonumber\\
=&-\frac{f'K}{\mathfrak{b}}\biggl[(1+c_3A)\left(\sum_{k}b^{kk}\right)\mathfrak{b}-\left(n+\frac{f}{f'K}\right)\mathfrak{b}^2\nonumber\\
&\qquad -\frac{nAu}{\sinh\rho}\mathfrak{b}+\mathfrak{b}H-n\biggr],\label{es-Q_2-5}
\end{align}
where we have thrown away the negative terms $-\frac{f}{\mathfrak{b}}$ and $-Af\sqrt{1+\frac{|\bar{\nabla}\rho|^2}{\sinh^2{\rho}}}$ in the second inequality of \eqref{es-Q_2-5}. Since
\begin{equation}\label{es-Q_2-6}
	\sum_k{b^{kk}}>\mathfrak{b},\quad \frac{f}{f'K}\leq \Theta,\quad \frac{u}{\sinh\rho}\leq 1,\quad H\geq\frac{n}{\mathfrak{b}},
\end{equation}
where $\Theta$ is the constant in item \eqref{p3} in Assumption \ref{ass}, substituting \eqref{es-Q_2-6} into \eqref{es-Q_2-5} gives the following estimates
\begin{equation}\label{es-Q_2-7}
	Q_2\leq -f'K\Big((1+c_3A-n-\Theta)\mathfrak{b}-nA\Big).
\end{equation}
Choosing
\begin{equation}\label{s3.A}
  A=\frac{n+\Theta}{c_3},
\end{equation}
which depends only on $n, M_0$ and $\Theta$, we see that if
\begin{equation}\label{con-b2}
	\mathfrak{b}>\frac{n(n+\Theta)}{c_3}
\end{equation}
then $Q_2<0$.

Combing \eqref{eq-sepa} with \eqref{con-b1} and \eqref{con-b2}, we know that if
 \begin{equation*}
 	\mathfrak{b}>\max\left\{\frac{Ac_4+\sqrt{A^2c_4^2+4}}{2},\frac{n(n+\Theta)}{c_3}\right\}=:a_1,
 \end{equation*}
where $A=\frac{n+\Theta}{c_3}$, then at the point $(\bar{p},\bar{t})$, we have
\begin{equation*}
	0\leq \partial_t Q-f'\dot{K}^{k\ell}\nabla_k\nabla_{\ell}{Q}=Q_1+Q_2<0,
\end{equation*}
which leads to a contradiction. Therefore we have
\begin{equation}\label{uppb}
	\mathfrak{b}(\bar{p},\bar{t})\leq~a_1.
\end{equation}
Since $(\bar{p},\bar{t})$ is a maximum point of the function $Q$, Combining \eqref{Q-initial} with \eqref{uppb}, we have
\begin{equation}\label{upp-Q}
	Q(p,t)\leq\max\left\{\log\max_{p\in M}{\mathfrak{b}(p,t_0)}+2Ac_2,\log a_1+2Ac_2\right\}
\end{equation}
for $(p,t)\in M\times [t_0,\min\{T,t_0+\tau\})$. Hence we have
\begin{align}\label{upperb-1}
	\mathfrak{b}(p,t)\leq& \max\left\{\max_{p\in M}{\mathfrak{b}(p,t_0)},a_1\right\}\mathrm{e}^{A(2c_2-\frac{c_1}{2})}\nonumber\\
=&:\max\left\{\max_{p\in M}{\mathfrak{b}(p,t_0)},a_1\right\}\mathrm{e}^{\Lambda_1}
\end{align}
for $(p,t)\in M\times [t_0,\min\{T,t_0+\tau\})$, where $\Lambda_1=\frac{n+\Theta}{a_1}(2c_2-\frac{c_1}{2})$ depending only on $n, M_0$ and $\Theta$.

Note that $t_0\in [t_0-\frac{\tau}{2},\min\{T,t_0+\frac{\tau}{2}\})$. Applying the above argument for the time interval $[t_0-\frac{\tau}{2},\min\{T,t_0+\frac{\tau}{2}\})$ gives
\begin{equation}\label{upperb-2}
	\max_{p\in M}{\mathfrak{b}(p,t_0)}\leq \max\left\{\max_{p\in M}{\mathfrak{b}(p,t_0-\frac{\tau}{2})},a_1\right\}\mathrm{e}^{\Lambda_1}.
\end{equation}
Combining \eqref{upperb-1} with \eqref{upperb-2} and the fact that $\mathrm{e}^{\Lambda_1}\geq 1$, we have
\begin{equation}
	\mathfrak{b}(p,t)\leq \max\left\{\max_{p\in M}{\mathfrak{b}(p,t_0-\frac{\tau}{2})},a_1\right\}\mathrm{e}^{2\Lambda_1}.
\end{equation}
By repeating the argument finitely many times, we finally get
\begin{align}
	\mathfrak{b}(p,t)&\leq \max\left\{\max_{p\in M}{\mathfrak{b}(p,0)},a_1\right\}\mathrm{e}^{\left(\left[\frac{2t_0}{\tau}\right]+2\right)\Lambda_1}\notag\\
	&\leq \max\left\{\max_{p\in M}{\mathfrak{b}(p,0)},a_1\right\}\mathrm{e}^{\left(\frac{2t}{\tau}+2\right)\Lambda_1}:=\Lambda_2\mathrm{e}^{\frac{2\Lambda_1}{\tau}t}
\end{align}
for all $(p,t)\in M\times [t_0,\min\{T,t_0+\tau\})$, where $[\cdot]$ denotes the integer part of a real constant, and $\Lambda_1,\Lambda_2:= \max\left\{\max_{p\in M}{\mathfrak{b}(p,0)},a_3\right\}\mathrm{e}^{2\Lambda_1}$ are constants depending only on $n,M_0$ and $\Theta$. Since $t_0$ is arbitrary, we conclude that the principal curvatures $\kappa_i$ of the solution $M_t$ of the flow \eqref{flow-VMCF} satisfy
\begin{equation*}
	\kappa_i\geq \Lambda_2^{-1}\mathrm{e}^{-\frac{2\Lambda_1}{\tau}t},\qquad i=1,\dots,n,
\end{equation*}
for all time $t\in [0,T)$. This completes the proof of Proposition \ref{preserve convex}.
\end{proof}
\begin{rem}
The auxiliary function $Q$ defined in \eqref{eq-Q} has also been used in \cite{LZ24} to derive a uniform positive lower bound of the principal curvatures along an anisotropic Gauss curvature type flow in the hyperbolic space.  
\end{rem}

\subsection{Upper bound of Gauss curvature} \label{sec.upK}
Now we prove the upper bound of the Gauss curvature for the solution $M_t$ along the flow \eqref{flow-VMCF}.
\begin{prop}\label{propKupp}
	Let $M_t,\ t\in[0,T)$ be the smooth solution of the flow \eqref{flow-VMCF} starting from a smooth closed convex hypersurface $M_0$. If $T<\infty$, then there exists a constant $C$ depending on $n, M_0, \Theta$ and $T$ such that the Gauss curvature $K$ of $M_t$ satisfies
	\begin{equation*}
		\max_{M_t}K\leq C
	\end{equation*}
	for any $t\in [0,T)$, where $\Theta$ is the constant in item \eqref{p3} of Assumption \ref{ass}.
\end{prop}
\proof
For any given $t_0\in[0,T)$, let $B_{\rho_0}(p_0)$ be the inball of $\Omega_{t_0}$ centered at some point $p_0\in \Omega_{t_0}$, where $\rho_0=\rho_{-}(t_0)$. Consider the support function $u(x,t)=\sinh \rho_{p_0}(x)\langle{\partial_{\rho_{p_0}},\nu}\rangle$ of $M_t$ with respect to the point $p_0$, where $\rho_{p_0}(x)$ is the distance function in $\mathbb{H}^{n+1}$ to the point $p_0$. Since $M_t$ is convex for all $t\in[0,T)$, by \eqref{u-bound} we have
\begin{equation}\label{equbound}
	2c\leq u\leq \sinh(2c_2)
\end{equation}
on $M_t$ for all $t\in \left[t_0,\min\{T,t_0+\tau\}\right)$. We define the auxiliary function as in \cite{Tso85}
\begin{equation*}
	W=\frac{f(K)}{u-c},
\end{equation*}
which is well-defined for $t\in \left[t_0,\min\{T,t_0+\tau\}\right)$. We shall apply the maximum principle to the evolution equation of $W$ to derive the upper bound of $K$.

Combining \eqref{eq-KK} and \eqref{equeq}, we compute that along the flow \eqref{flow-VMCF} the function $W$ evolves as
\begin{align}\label{eqW}
	\frac{d}{d t}{W}=&f'\dot{K}^{ij}\left(W_{ij}+\frac{2}{u-c}u_iW_j\right)\nonumber\\
&-\frac{\phi(t)}{u-c}\Big(f'(HK-\sigma_{n-1}(\kappa))+W\cosh{\rho_{p_0}}(x)\Big)\nonumber\\
	&+\frac{f}{(u-c)^2}(f+nf'K)\cosh{\rho_{p_0}}(x)\nonumber\\
&-\frac{cf}{(u-c)^2}f'HK-Wf'\sigma_{n-1}(\kappa)\nonumber\\
	\leq& f'\dot{K}^{ij}\left(W_{ij}+\frac{2}{u-c}u_iW_j\right)+\frac{\phi(t)}{u-c}f'\sigma_{n-1}(\kappa)\nonumber\\
	&+(1+n\frac{f'K}{f})W^2\cosh{\rho_{p_0}}(x)-c\frac{f'K}{f}HW^2.
\end{align}

Let	$\widetilde{W}(t)=\max_{M_t}W(x,t)$. Noting that $f(K)=(u-c)W$, by the definition \eqref{eqphi} and the upper bound \eqref{equbound} of $u$, we have:
\begin{align*}
		\phi(t)=&\frac{1}{|M_t|}\int_{M_t}f(K)d\mu_t\\
\leq&\max_{M_t}f(K(\cdot,t))	\leq(\sinh(2c_2)-c)\widetilde{W}.
\end{align*}
By the lower bound on the principal curvatures in Lemma \ref{preserve convex}, we also have
\begin{align}\label{s3.sigman-1}
	\sigma_{n-1}(\kappa)=&K(\frac{1}{\kappa_1}+\cdots\frac{1}{\kappa_n})\nonumber\\
\leq &nK(\min_{1\leq i\leq n} \kappa_i)^{-1}\leq nK\Lambda_2\mathrm{e}^{\frac{2\Lambda_1}{\tau}T}.
\end{align}
 It follows that
 \begin{equation}\label{s4.2-1}
 	\frac{\phi(t)}{u-c}f'\sigma_{n-1}(\kappa)\leq ~n (\sinh(2c_2)-c)\Lambda_2\mathrm{e}^{\frac{2\Lambda_1T}{\tau}}\frac{f'K}{f}\widetilde{W}^2.
 \end{equation}
Since $H\geq nK^{1/n}$, the last term of \eqref{eqW} satisfies
\begin{equation}\label{s4.2-3}
	-c\frac{f'K}{f}HW^2\leq -nc\frac{f'K^{\frac{n+1}{n}}}{f}W^2.
\end{equation}
Substituting \eqref{s4.2-1} and \eqref{s4.2-3} into \eqref{eqW}, we arrive at
\begin{align}\label{s4.2-4}
 \frac{d}{dt}\widetilde{W}\leq &\widetilde{W}^2\cosh(2c_2)+n\widetilde{W}^2\frac{f'K}{f}\Big( (\sinh(2c_2)-c)\Lambda_2\mathrm{e}^{\frac{2\Lambda_1T}{\tau}}\nonumber\\
 &\quad +\cosh(2c_2)-cK^{1/n}\Big).
\end{align}

Denote
\begin{equation*}
  \bar{c}=(nc)^{-n}\left((\cosh(2c_2)+1)\Theta+n(\sinh(2c_2)-c)\Lambda_2\mathrm{e}^{\frac{2\Lambda_1T}{\tau}}+n\cosh(2c_2)\right)^n,
\end{equation*}
which depends on $n, M_0, \Theta$ and $T<\infty$. Suppose that
\begin{equation*}
  \widetilde{W}(t)\geq \frac{f(\bar{c})}{c}.
\end{equation*}
Then for any $x\in M_t$ with $W(x,t)=\widetilde{W}(t)$, there holds
\begin{align*}
  f(K(x,t))=&W(x,t)(u-c)\nonumber\\
  =&\widetilde{W}(t)(u-c)\geq f(\bar{c}).
\end{align*}
As $f$ is strictly increasing, this implies that
\begin{equation}\label{s3.Klbd}
  K(x,t)\geq \bar{c}
\end{equation}
for any $x\in M_t$ with $W(x,t)=\widetilde{W}(t)$. Substituting \eqref{s3.Klbd} into \eqref{s4.2-4} and noting that $f'K/f\geq 1/\Theta$ by item \eqref{p3} of Assumption \ref{ass}, we have
\begin{equation}\label{s4.2-5}
 \frac{d}{dt}\widetilde{W}(t)\leq -\widetilde{W}^2(t)
\end{equation}
whenever $\widetilde{W}(t)\geq {f(\bar{c})}/{c}$. It follows that
\begin{equation}\label{s4.2-7}
\widetilde{W}(t)\leq\max\left\{ \frac{1}{\widetilde{W}^{-1}(t_0)+t-t_0},\frac{f(\bar{c})}{c}\right\}
\end{equation}
for all time $t\in \left[t_0,\min\{T,t_0+\tau\}\right)$.

For $t_0=0$, we obtain from \eqref{s4.2-7} the upper bound
\begin{equation*}
	\widetilde{W}(t)\leq\max\left\{\widetilde{W}(0),\frac{f(\bar{c})}{c}\right\},\quad \forall t\in[0,\min\{\tau,T\})
\end{equation*}
and so
\begin{align}\label{s4.2-8}
	f(K)=&(u-c)W\nonumber\\
\leq &\sinh(2c_2)\max\left\{\widetilde{W}(0),\frac{f(\bar{c})}{c}\right\},\quad \forall t\in[0,\min\{\tau,T\}).
\end{align}
Next, for $t_0=\tau/2$, the estimate \eqref{s4.2-7} implies
\begin{align*}
\widetilde{W}(t)&\leq\max\left\{\frac{1}{\widetilde{W}^{-1}(t_0)+t-t_0},\frac{f(\bar{c})}{c}\right\}\\
&\leq\max\left\{\frac{1}{t-t_0},\frac{f(\bar{c})}{c}\right\}\\
&\leq \max\left\{\frac{2}{\tau},\frac{f(\bar{c})}{c}\right\}
\end{align*}
for $t\in [\tau, \min\{3\tau/2,T\})$, and so
\begin{equation}\label{s4.2-9}
	f(K)\leq \sinh(2c_2)\max\left\{\frac{2}{\tau},\frac{f(\bar{c})}{c}\right\}
\end{equation}
for $t\in [\tau, \min\{3\tau/2,T\})$. Repeating the above argument for $t_0=m\tau/2(m\geq 2)$, we can get the estimate \eqref{s4.2-9} for $t\in [\frac{(m+1)\tau}2, \min\{\frac{(m+2)\tau}2,T\})$, which covers the whole time interval $[0,T)$.

Combining \eqref{s4.2-8}, \eqref{s4.2-9} and the fact that $f$ is an increasing function with $\lim\limits_{x\to\infty}f(x)=+\infty$ (see items \eqref{p1} and \eqref{p2} in Assumption \ref{ass}), we obtain the upper bound $K\leq C$ for a constant $C$ depending on $n, M_0, \Theta$ and $T$. This completes the proof of Proposition \ref{propKupp}.
\endproof

\section{Long time existence}\label{sec5}
In this section, we prove the long time existence of the flow \eqref{flow-VMCF}. We need to show that the solution remains smooth as long as the curvatures are bounded. To prove this, we need  the following result which is a special case of Theorem 6 in \cite{And2004} (see also \cite[Theorem 3.2]{AleSin10b}).
\begin{thm}[\cite{And2004}]\label{s4.thm1}
	Let $\Omega$ be a domain in $\mathbb{R}^n$. Let $u\in C^4(\Omega\times [0,T))$ be a function satisfying
	\begin{equation*}
		\frac{\partial}{\partial t}u=F(D^2u,Du,u,x,t),
	\end{equation*}
	where $F$ is $C^2$ and is elliptic, i.e., $\lambda I\leq (\dot{F}^{ij})\leq \Lambda I$ for some constants $\Lambda>\lambda>0$. Suppose that $F$ can be written as $F=\varphi(G(D^2u,Du,u,x,t))$, where $G$ is concave with respect to $D^2u$ and $\varphi$ is an increasing function on the range of $G$. Then in any relatively compact $\Omega'\subset\Omega$ and for any $\tau\in (0,T)$ we have
	\begin{equation*}
		\|u\|_{C^{2,\beta}(\Omega'\times (\tau,T))}\leq C,
	\end{equation*}
	where $\beta\in (0,1)$ depends on $n,\lambda$ and $\Lambda$, and $C$ depends on $\lambda,\Lambda,  \|u\|_{C^{2}(\Omega\times [0,T))}$, $\tau, \mathrm{dist}(\Omega',\partial\Omega)$ and the bounds on the first and second derivatives of $G$.
\end{thm}
The advantage of the above theorem is that it allows to relax the concavity hypothesis of the usual regularity theorem for fully nonlinear parabolic equation.

\begin{thm}\label{long}
	Let $M_0$ be a smooth closed convex hypersurface in $\mathbb{H}^{n+1}$ and $M_t$ be the smooth solution of the flow \eqref{flow-VMCF} starting from $M_0$ with $\phi(t)$ given by \eqref{eqphi} and the speed function satisfies the assumption \ref{ass}. Then $M_t$ remains convex and exists for all time $t\in[0,\infty)$.
\end{thm}
\proof
We will argue by contradiction. Let $[0,T)$ be the maximal interval such that the solution of the flow \eqref{flow-VMCF} exists with $T<\infty$. Then combining Proposition \ref{preserve convex} and Proposition \ref{propKupp} yields that the principal curvatures $\kappa=(\kappa_1,\dots,\kappa_n)$ of $M_t$ satisfy
\begin{equation}\label{s4.cur}
	0<\underline{\kappa}_0\leq \kappa_i\leq \overline{\kappa}_0,\quad i=1,\dots,n
\end{equation}
for all time $t\in [0,T)$, where the constants $\underline{\kappa}_0, \overline{\kappa}_0$ depend on $n, M_0, \Theta$ and $T$.

To prove the long time existence of the solution $M_t$ of the flow \eqref{flow-VMCF}, we need to derive the higher order regularity estimates. Recall that up to a tangential diffeomorphism, the flow equation \eqref{flow-VMCF} is equivalent to the following scalar parabolic equation
\begin{equation}\label{s4.eqrho}
	\frac{\partial}{\partial t}\rho=\left(\phi(t)-f(K)\right)\sqrt{1+\frac{|\bar{\nabla}\rho|^2}{\sinh^2{\rho}}},
\end{equation}
of the radial graph function $\rho$ over $\mathbb{S}^n$, where $K=K(\bar{\nabla}^2\rho,\bar{\nabla}\rho,\rho)$ is expressed in \eqref{eq-Gauss} and $\bar{\nabla}$ denotes the Levi-Civita connection with respect to the round metric on $\mathbb{S}^n$. Denote the right hand side of \eqref{s4.eqrho} by $F[\bar{\nabla}^2\rho,\bar{\nabla}\rho,\rho,t]$.

For any $t_0\in [0,T)$, we consider the solution of $M_t$ in the time interval $[t_0, \min\{t_0+\tau,T\})$.  Since only $f(K)$ in $F[\bar{\nabla}^2\rho,\bar{\nabla}\rho,\rho,t]$ depends on $\bar{\nabla}^2\rho$, we calculate that
\begin{align}
	\dot{F}^{ij}=&-f'\frac{\partial K}{\partial \rho_{ij}}\sqrt{1+\frac{|\bar{\nabla}\rho|^2}{\sinh^2{\rho}}}\nonumber\\
=&\frac{f'\sinh\rho}{(\sinh^2\rho+|\bar{\nabla} \rho|^2)^{\frac{n+2}{2}}(\sinh \rho)^{2(n-1)}}\frac{\partial \det h_{ij}}{\partial h_{ij}}\sqrt{1+\frac{|\bar{\nabla}\rho|^2}{\sinh^2{\rho}}},
\end{align}
where $h_{ij}$ is expressed in \eqref{ex-h}.  The $C^0, C^1$ estimates obtained in \eqref{rhobound}, \eqref{s3.C1} and the curvature bound \eqref{s4.cur} implies that $F$ is elliptic, i.e., $\lambda I\leq (\dot{F}^{ij})\leq \Lambda I$ for some constants $\Lambda>\lambda>0$ depending on $n, M_0, \Theta$ and $T$.  Moreover, since it's well known that $K^{1/n}$ is a concave operator with respect to second derivatives, and $f(K)$ is a strictly increasing function of $K^{1/n}$, by Theorem \ref{s4.thm1} we derive a $C^{2,\alpha}$ estimate on $\rho$, see also the arguments in \cite{CS10,Mc05} for the $C^{2,\gamma}$ estimate of the solutions to volume preserving curvature flows. Then by the parabolic Schauder theory (see \cite{Lie96}), we can deduce all higher order regularity estimates of $\rho$ on $[t_0, \min\{t_0+\tau,T\})$. As $t_0$ is arbitrary, we can obtain the smoothness of the flow for all time $t\in [0,T)$ and a standard continuation argument then shows that $T=+\infty$.
\endproof

\begin{rem}
	Note that the curvature estimate \eqref{s4.cur} of the solution $M_t$ of the flow \eqref{flow-VMCF} depends on time $t$ and may degenerate as time $t\to\infty$. To study the asymptotical behavior of $M_t$ as $t\to\infty$, we still need to get an uniform curvature estimate which does not depend on time. This will be obtained in the next two sections.
\end{rem}

\section{Hausdorff convergence}\label{sec.hau}
In this section, we prove the monotonicity of the quermassintegral $\mathcal{A}_{n-1}(\Omega_t)$, the subsequential Hausdorff convergence of the solution $M_t$ of \eqref{flow-VMCF} and the convergence of the center of the inner ball of $\Omega_t$ to a fixed point.

Denote the average integral of the Gauss curvature by
\begin{equation}\label{s6.0}
	\bar{K}=\frac{1}{|M_t|}\int_{M_t}{K d\mu_t}=\frac{\mathcal{A}_n(\Omega_t)+\frac{1}{n-1}\mathcal{A}_{n-2}(\Omega_t)}{\mathcal{A}_0(\Omega_t)},
\end{equation}
where the second equality is due to \eqref{eq-VW}. It follows from the monotonicity \eqref{s2.Akmo} of quermassintegrals with respect to inclusion of convex sets and the estimates on inner radius and outer radius in Lemma \ref{Lemma-ioradius} that there exists uniform positive constants $m_1, m_2, a$ and $b$ depending only on $n$ and $M_0$, such that
\begin{equation}\label{In-lubbK}
  0<m_1\leq |M_t|=A_0(\Omega_t)\leq m_2,\quad   a\leq\bar{K}\leq b.	
\end{equation}

\subsection{Monotonicity for $\mathcal{A}_{n-1}$}
\begin{lem}\label{lemmono}
Let $M_t$ be a smooth convex solution of the volume preserving flow \eqref{flow-VMCF}. Denote $\Omega_t$ the domain enclosed by $M_t$. Then $\mathcal{A}_{n-1}(\Omega_t)$ is monotone decreasing in time $t$, which is strictly decreasing unless $\Omega_t$ is a geodesic ball.
\end{lem}
\proof
From the evolution equation \eqref{eqWk} for the quermassintegrals of $\Omega_t$, we have
\begin{equation*}
	\frac{d}{dt}\mathcal{A}_{n-1}(\Omega_t)=n\int_{M_t}{K(\phi(t)-f(K))d\mu_t}.
\end{equation*}
Since $\phi(t)$ is defined as in \eqref{eqphi}, we have
\begin{align}
	\frac{d}{dt}\mathcal{A}_{n-1}(\Omega_t)&=\frac{n}{|M_t|}\left(\int_{M_t}K d\mu_t\int_{M_t}f(K) d\mu_t-|M_t|\int_{M_t}Kf(K) d\mu_t\right)\notag\\
	&=n\int_{M_t}{f(K)(\bar{K}-K)\,d\mu_t}\notag\\
	&=-n\int_{M_t}{(f(K)-f(\bar{K})(K-\bar{K}))}\,d\mu_t\leq 0\label{eqWmo}
\end{align}
due to the assumption that $f$ is strictly increasing. Moreover, equality holds in \eqref{eqWmo} at some time $t$ if and only if $K$ is a constant on $M_t$, which means that $M_t$ is a geodesic sphere by the Alexandrov type theorem for hypersurfaces with constant Gauss curvature in the hyperbolic space (see \cite{MS91}).
\endproof

\subsection{Subsequential Hausdorff convergence}
We first prove the following estimate on the $L^1$ oscillation decay of Gauss curvature:
\begin{lem}\label{Lem-subsec-6.2-1}
	Let $M_0$ be a smooth closed and convex hypersurface in $\mathbb{H}^{n+1}$ and $M_t$ be the smooth solution of the flow \eqref{flow-VMCF} starting from $M_0$. Then there exists a sequence of times $\{t_i\},t_i\to\infty$, such that
	\begin{equation}\label{s6.2-1}
		\int_{M_{t_i}}{|K-\bar{K}|d\mu_{t_i}}\to 0,\quad \text{as}\,\, t_i\to\infty.
	\end{equation}
\end{lem}
\begin{proof}
	By the evolution equation \eqref{eqWmo} and the long time existence of the flow \eqref{flow-VMCF}, we have
	\begin{equation*}
		n\int_0^{\infty}\int_{M_t}{(f(K)-f(\bar{K}))(K-{\bar{K}})d\mu_t}dt\leq {\mathcal{A}_{n-1}}(\Omega_0)<\infty.
	\end{equation*}
	Therefore there exists a sequence of times $t_i\to\infty$ such that
	\begin{equation}\label{s6.2-2}
		\int_{M_{t_i}}{(f(K)-f(\bar{K}))(K-{\bar{K}})d\mu_{t_i}}\to 0.
	\end{equation}
Denote the subset $Y_t\subset M_t$ as $Y_t=\{p\in M_t|K(p)=\bar{K}\}$. Then we have
\begin{align}
	\int_{M_t}{|K-\bar{K}|}\,d\mu_t&=\int_{M_t\setminus Y_t}{\frac{|K-\bar{K}|^{\frac{1}{2}}}{|f(K)-f(\bar{K})|^{\frac{1}{2}}}|K-\bar{K}|^{\frac{1}{2}}|f(K)-f(\bar{K})|^{\frac{1}{2}}}\,d\mu_t\notag\\
	&\leq \underbrace{\left(\int_{M_t\setminus Y_t}{\frac{|K-\bar{K}|}{|f(K)-f(\bar{K})|}}\,d\mu_t\right)^{\frac{1}{2}}}_{(I)}\left(\int_{M_t}{(f(K)-f(\bar{K}))(K-{\bar{K}})d\mu_t}\right)^{\frac{1}{2}}.\label{In-KKbar}
\end{align}

Next, we show that the term $(I)$ of \eqref{In-KKbar} is uniformly bounded from above. We divide the set $M_t\setminus Y_t$ into three disjoint subsets
\begin{equation*}
	M_t\setminus Y_t=Z_1\sqcup Z_2\sqcup Z_3,
\end{equation*}
where the subsets $Z_1, Z_2$ and $Z_3$ are defined as
\begin{align}
	Z_1&=\{p\in M_t\setminus Y_t|0<K(p)\leq\frac{a}{2}\},\label{dn-T_1}\\
	Z_2&=\{p\in	M_t\setminus Y_t|\frac{a}{2}< K(p)< \frac{3b}{2}\},\label{dn-T_2}\\
	Z_3&=\{p\in	M_t\setminus Y_t|K(p)\geq\frac{3b}{2}\}.\label{dn-T_3}
\end{align}
Using \eqref{In-lubbK}, we calculate as follows:
\begin{itemize}
	\item[(i)] If $p\in Z_1$, since $f$ is an increasing function of $K$, we have
	\begin{equation*}
		\frac{|K-\bar{K}|}{|f(K)-f(\bar{K})|}=\frac{\bar{K}-K}{f(\bar{K})-f(K)}\leq \frac{b}{f(a)-f(\frac{a}{2})}
	\end{equation*}
and hence
   \begin{equation}\label{sec6-case-1}
   	\int_{Z_1}{\frac{|K-\bar{K}|}{|f(K)-f(\bar{K})|}}\,d\mu_t\leq \frac{m_2b}{f(a)-f(\frac{a}{2})}.
   \end{equation}
  \item[(ii)] If $p\in Z_2$, then by Lagrange's mean value theorem, we have
  \begin{align*}
  	\frac{|K-\bar{K}|}{|f(K)-f(\bar{K})|}&=\frac{1}{f'(\xi)},\quad \text{the value $\xi$ is taken between $K$ and $\bar{K}$}\\
  	&\leq \max_{x\in[\frac{a}{2},\frac{3b}{2}]}\left\{\frac{1}{f'(x)}\right\}=:c
  \end{align*}
and hence
\begin{equation}\label{sec6-case-2}
	\int_{Z_2}{\frac{|K-\bar{K}|}{|f(K)-f(\bar{K})|}}\,d\mu_t\leq m_2c.
\end{equation}
   \item[(iii)] If $p\in Z_3$, we have
   \begin{equation*}
   	\frac{|K-\bar{K}|}{|f(K)-f(\bar{K})|}=\frac{K-\bar{K}}{f(K)-f(\bar{K})}\leq\frac{K}{f(\frac{3b}{2})-f(b)}
   \end{equation*}
and hence
\begin{align}\label{sec6-case-3}
	\int_{Z_3}{\frac{|K-\bar{K}|}{|f(K)-f(\bar{K})|}}\,d\mu_t&\leq \frac{1}{f(\frac{3b}{2})-f(b)}\int_{M_t}{K}\,d\mu_t\notag\\
	&=\frac{|M_t|\bar{K}}{f(\frac{3b}{2})-f(b)}\notag\\
	&\leq\frac{m_2b}{f(\frac{3b}{2})-f(b)}.
\end{align}
\end{itemize}
Combining \eqref{sec6-case-1}-\eqref{sec6-case-3}, we conclude that the term $(I)$ of \eqref{In-KKbar} is uniformly bounded from above. Then by \eqref{s6.2-2}, we complete the proof of Lemma \ref{Lem-subsec-6.2-1}.
\end{proof}

With the estimate \eqref{s6.2-1} in hand, we can argue as in \cite[Lemma 6.3]{WYZ2022} using curvature measure theory for convex bodies in $\mathbb{H}^{n+1}$ to get the subsequential Hausdorff convergence of $M_t$ to a geodesic sphere. We state the result in the following lemma and refer the readers to \cite[Lemma 6.3]{WYZ2022} for the proof.

\begin{lem}\label{subcon}
	Let $M_0$ be a smooth, closed convex hypersurface in $\mathbb{H}^{n+1}$ and $M_t$ be the smooth solution of the flow \eqref{flow-VMCF} starting from $M_0$. Then there exists a sequence of times $\{t_i\},t_i\to\infty$, such that $M_{t_i}$ converges to a geodesic sphere $S_{\rho_{\infty}}(p)$ in Hausdorff sense as $t_i\to\infty$, where $p$ is the center of the sphere and the radius $\rho_{\infty}$ is determined by the fact that $S_{\rho_{\infty}}(p)$ encloses the same volume of $M_0$.
\end{lem}

\subsection{Convergence of the center of the inner ball}As in \cite{WYZ2022}, since we do not have the analogous stability estimate as in \cite[Eq.(7.124)]{RS2014} for the hyperbolic case, we can not apply the argument in \cite{AW21} to deduce from Lemma \ref{subcon} the Hausdorff convergence of $M_t$ to the geodesic sphere for all time $t\to \infty$. However, if we denote $p_t$ as the center of the inner ball of $\Omega_t$, we can still prove that $p_t$ converges to the fixed point $p\in \mathbb{H}^{n+1}$ for all time $t\to\infty$ using the Alexandrov reflection and the subsequential Hausdorff convergence of $M_t$ in Lemma \ref{subcon}.

Let $p\in \mathbb{H}^{n+1}$ be the center of the limit geodesic sphere $S_{\rho_{\infty}}(p)$ in Lemma \ref{subcon}. Take an arbitrary direction $z\in T_p\mathbb{H}^{n+1}$. Let $\gamma_z$ be the normal geodesic line (i.e. $|\gamma'|=1$) through the point $p$ with $\gamma_z(0)=p$ and $\gamma'_z(0)=z$, and let $H_{z,s}$ be the totally geodesic hyperplane in $\mathbb{H}^{n+1}$ that is perpendicular to $\gamma_z$ at $\gamma_z(s), s\in\mathbb{R}$. We use the notation $H_{z,s}^{+}$ and $H_{z,s}^{-}$ for the half-spaces in $\mathbb{H}^{n+1}$ determined by $H_{z,s}$ as follows:
\begin{equation*}
	H_{z,s}^{+}:=\bigcup_{s'\geq s}H_{z,s'},\qquad  H_{z,s}^{-}:=\bigcup_{s'\leq s}H_{z,s'}.
\end{equation*}
For a bounded domain $\Omega$ in $\mathbb{H}^{n+1}$, denote
\begin{equation*}
	\Omega_z^{+}(s)=\Omega\cap H_{z,s}^{+},\qquad \Omega_z^{-}(s)=\Omega\cap H_{z,s}^{-}.
\end{equation*}
The reflection map across $H_{z,s}$ is denoted by $R_{\gamma_z,s}$. We define
\begin{align*}
	S_{\gamma_z}^{+}(\Omega)&:=\inf\{s\in\mathbb{R}~|~R_{\gamma_z,s}(\Omega_z^{+}(s))\subset\Omega_z^{-}(s)\},\\
	S_{\gamma_z}^{-}(\Omega)&:=\sup\{s\in\mathbb{R}~|~R_{\gamma_z,s}(\Omega_z^{-}(s))\subset\Omega_z^{+}(s)\}.
\end{align*}
The Alexandrov reflection argument implies that $S_{\gamma_z}^{+}(\Omega_t)$ is non-increasing in $t$ for each $z$ (see \cite[Lemma 6.1]{BenWei}). By the definitions of $S_{\gamma_z}^{+}(\Omega_t)$ and $S_{\gamma_z}^{-}(\Omega_t)$, we have $S_{\gamma_z}^{-}(\Omega_t)\leq S_{\gamma_z}^{+}(\Omega_t)$. Since $S_{\gamma_z}^{-}(\Omega_t)=-S_{\gamma_{-z}}^{+}(\Omega_t)$, we also have that $S_{\gamma_z}^{-}(\Omega_t)$ is non-decreasing in $t$ for each $z$. Note that the paper \cite{BenWei} deals with the flow with $h$-convex initial hypersurfaces, the argument in Lemma 6.1 of \cite{BenWei} works for convex solutions as well. The readers may refer to \cite{Chow97,CG96,Chow23} for more details on applications of the Alexandrov reflection method in extrinsic curvature flows.

The subsequential Hausdorff convergence in Lemma \ref{subcon} implies that there exists a sequence of  $d_i\rightarrow 0$ such that $M_{t_i}\subset B_{\rho_{\infty}+d_i}(p)/ B_{\rho_{\infty}-d_i}(p)$. Take an arbitrary direction $z\in T_p{\mathbb{H}^{n+1}}$, we have shown in  \cite{WYZ2022}*{Lemma 6.6} that 
\begin{align}\label{s5.refl2-0}
    -C\sqrt{d_i}\leq S_{\gamma_z}^{-}(\Omega_{t_i})\leq S_{\gamma_z}^{+}(\Omega_{t_i})\leq C\sqrt{d_i}
\end{align}
for some $C=C(\rho_\infty)$. The monotonicity of $S_{\gamma_z}^{+}$ and $S_{\gamma_z}^{-}$ then implies 
\begin{align}\label{s5.refl2}
    -C\sqrt{d_i}\leq S_{\gamma_z}^{-}(\Omega_t)\leq S_{\gamma_z}^{+}(\Omega_{t})\leq C\sqrt{d_i},\qquad \forall t\geq t_i.
\end{align}

Let $x_1, x_2\in M_t$ be the points such that $d(x_1,p)=\max_{x\in M_t}d(x,p)$ and $d(x_2,p)=\min_{x\in M_t}d(x,p)$. There exists a geodesic $\gamma_z$ passing through $p$ (with $\gamma_z(0)=p$) and totally geodesic hyperplane $H_{z,C\sqrt{d_i}}$ orthogonal to both $\gamma_z$ and the geodesic connecting $x_1$ to  $x_2$. Using $H_{z,C\sqrt{d_i}}$ as the reflecting hyperplane and noting \eqref{s5.refl2}, we can estimate that $\max_{x\in M_t}d(x,p)-\min_{x\in M_t}d(x,p)\leq Cd_i^{1/4}$ for any $t\geq t_i$. This implies that 
\begin{align}\label{s5.distH}
    \text{d}_\mathcal{H}(\partial\Omega_t, \partial B_{\rho_{-}(t)}(p_t))\leq Cd_i^{1/4},\qquad \forall~t\geq t_i,
\end{align}
where we denote $p_t$ as the center of an inner ball $B_{\rho_{-}(t)}(p_t)$ of $\Omega_t$.  

Consider another geodesic (still denoted by $\gamma_z$) which passes through $p$ and $p_t$ with $\gamma_z(0)=p$ and $\gamma_z(s_t)=p_t$.  We claim that 
\begin{align}\label{s5.refl1}
    S_{\gamma_z}^{-}(\Omega_t)-\text{d}_\mathcal{H}(\partial\Omega_t, \partial B_{\rho_{-}(t)}(p_t))\leq~s_t\leq ~S_{\gamma_z}^{+}(\Omega_t)+\text{d}_\mathcal{H}(\partial\Omega_t, \partial B_{\rho_{-}(t)}(p_t)).
\end{align}
In fact, denote $S_{\gamma_z}^{+}(\Omega_t)$ by $\bar{s}_t$. If $p_t\in H^-_{z,\bar{s}_t}$, we obviously have $s_t\leq S_{\gamma_z}^{+}(\Omega_t)$;  If $p_t\in H^+_{z,\bar{s}_t}$, we have $B_{\rho_{-}(t)}(p'_t):=R_{\gamma_z,\bar{s}_t}(B_{\rho_{-}(t)}(p_t))\subset\Omega_t$. It follows that
\begin{align*}
 2(s_t-S_{\gamma_z}^{+}(\Omega_t))= &\text{d}_\mathcal{H}(\partial B_{\rho_{-}(t)}(p'_t), \partial B_{\rho_{-}(t)}(p_t))\\
 \leq &  2\text{d}_\mathcal{H}(\partial\Omega_t, \partial B_{\rho_{-}(t)}(p_t)).
\end{align*}
This proves the second inequality of \eqref{s5.refl1}. The proof of the first inequality of \eqref{s5.refl1} is similar. 

Finally, combining \eqref{s5.refl2} -- \eqref{s5.refl1}, we have
\begin{align*}
    d(p_t,p)=|s_t|\leq &~Cd_i^{1/2}+Cd_i^{1/4},\qquad \forall~t\geq t_i,
\end{align*}
which implies that $d(p_t,p)\to 0$ as $t\to\infty$.

\section{Smooth convergence}\label{Sec-main proof}
In this section, we complete the proof of Theorem \ref{theo}. Firstly, we prove the following uniform estimate for the principal curvatures of $M_t$ along the flow \eqref{flow-VMCF}.
\begin{lem}\label{uni}
	Let $M_0$ be a smooth, closed and convex hypersurface in $\mathbb{H}^{n+1}$, and $M_t$ be the smooth solution of the flow \eqref{flow-VMCF} starting from $M_0$. Then there exists constants $\underline{\kappa}$, $\overline{\kappa}$ depending only on $n, M_0$ and $\Theta$ such that the principal curvatures $\kappa_i$ of $M_t$ satisfy:
	\begin{equation}
		\underline{\kappa}\leq \kappa_i\leq \overline{\kappa},\qquad i=1,\dots,n
	\end{equation}
	for all time $t\in [0,+\infty)$, where $\Theta$ is the constant in item \eqref{p3} of Assumption \ref{ass}.
\end{lem}
\proof
Since the center $p_t$ of an inner ball of $\Omega_t$ converges to a fixed point $p$ as $t\to\infty$ and the inner radius of $\Omega_t$ has a positive lower bound $\rho_-(t)\geq c_1$, there exists a sufficiently large time $t^{*}$, depending on $c_1$ and hence depending only on $n$ and $M_0$, such that $d(p_t,p)<{c_1}/{4}$ for $t\geq t^{*}$. Then we have:
\begin{equation}\label{s7.0}
	B_{c_1/4}(p)\subset \Omega_t,\qquad \forall~t\geq t^{*}.
\end{equation}
Applying Proposition \ref{preserve convex} to the time interval $[0,t^*)$ and $[t^*,\infty)$ respectively gives a uniform lower bound for the principal curvatures of $M_t$ for all time $t>0$. In fact, on the time interval $[0,t^*)$, the estimate \eqref{s4.2-0} implies that the principal curvatures $\kappa_i$ of $M_t$ satisfy
\begin{equation}\label{s7.1}
	k_i\geq \Lambda_2^{-1}\mathrm{e}^{-\frac{2\Lambda_1}{\tau}t^*},\qquad t\in[0,t^*).
\end{equation}
While for time $t\in [t^*,\infty)$, since $B_{c_1/4}(p)\subset \Omega_t$ for all time $t\in [t^*,\infty)$. Then by the estimate \eqref{upperb-1} in the proof of Proposition \ref{preserve convex}, we have
\begin{equation*}
	\mathfrak{b}(p,t)\leq \max\left\{\max_{p\in M}{\mathfrak{b}(p,t^*)},a_3\right\}\mathrm{e}^{\Lambda_1}
\end{equation*}
for all $(p,t)\in M\times[t^*,\infty)$. This together with \eqref{s7.1} implies that the principal curvatures of $M_t$ are uniformly bounded from below by a positive constant  $\underline{\kappa}$ which depends only on $n, M_0$ and $\Theta$.

Once we have the uniform lower bound for the principal curvatures, the uniform upper bound for the Gauss curvature $K$ follows easily from the proof of Proposition \ref{propKupp}. In fact, the upper bound \eqref{s3.sigman-1} for $\sigma_{n-1}(\kappa)$ in the proof of Proposition \ref{propKupp} now has the form
\begin{equation*}
  \sigma_{n-1}(\kappa)\leq \frac{n}{\underline{\kappa}}K,
\end{equation*}
where the coefficient of $K$ does not depend on time $t$.

Therefore, combining the uniform upper bound on $K$ and the lower bound $\kappa_i\geq \underline{\kappa}$,  there exists a constant $\overline{\kappa}$ such that $\kappa_i\leq \overline{\kappa}$ for all $i=1,\dots,n$. This completes the proof of Lemma \ref{uni}.
\endproof

It follows from Lemma \ref{uni} that the flow \eqref{flow-VMCF} is uniformly parabolic for all time $t>0$. Then an argument similar to that in the proof of Theorem \ref{long} can be applied to show that all derivatives of curvatures are uniformly bounded on $M_t$ for all $t>0$. This together with Lemma \ref{subcon} implies there exists a sequence of times $t_i\to\infty$, such that $M_{t_i}$ converges smoothly to a geodesic sphere $S_{\rho_{\infty}}(p)$ as $t_i\to\infty$.

The full time convergence and the exponential convergence can be obtained by studying the linearization of the flow \eqref{flow-VMCF}. For each sufficiently large time $t_k$, we write $M_{t_k}$ as the graph of the radial function $\rho_{t_k}(\cdot)$ over $\mathbb{S}^n$ centered at $p_{t_k}$. For time $t$ sufficiently close to $t_k$, we rewrite the flow equation \eqref{flow-VMCF} as the scalar parabolic PDE
\begin{equation}\label{eq-rho}
	\left\{\begin{aligned}
		\partial_t\rho&=\left(\phi(t)-f(K)\right)\sqrt{1+{|\bar{\nabla}\rho|^2}/{\sinh^2\rho}},\quad t>t_k\\
		\rho(\cdot&,t_k)=\rho_{t_k}(\cdot).
	\end{aligned}\right.
\end{equation}
Note we can assume the oscillation of $\rho_{t_k}-\rho_{\infty}$ is sufficiently small by choosing $t_k$ large enough. By a direct computation, the linearized equation of the flow \eqref{eq-rho} about the geodesic sphere of radius $\rho_{\infty}$ is given by
\begin{equation}\label{eqdeta2}
	\frac{\partial }{\partial t}\eta=\frac{ (\coth\rho_{\infty})^{n-1}f'}{\sinh^2\rho_{\infty}}\left(\bar{\Delta}\eta+n\eta-\frac{n}{|\mathbb{S}^n|}\int_{\mathbb{S}^n}{\eta\, d\sigma}\right).
\end{equation}
where $f'$ is taken the value at the point $x=(\coth\rho_{\infty})^n$.

Since the oscillation of $\rho_{t_k}-\rho_{\infty}$ is sufficiently small, it follows exactly in \cite{Cab-Miq2007}, using \cite{Esc98}, that the solution $\rho(\cdot,t)$ of \eqref{eq-rho} starting at $\rho_{t_k}(\cdot)$ exists for all time and converges exponentially to a constant $\rho_{\infty}$. This means that the hypersurface $\overline{M}_t=$ graph $\rho(\cdot,t)$ solves \eqref{flow-VMCF} with initial condition $M_{t_k}$ and by uniqueness $\overline{M}_t$ coincides with $M_t$ for $t\geq t_k$, and hence the solution $M_t$ of \eqref{flow-VMCF} with initial condition $M_0$ converges exponentially as $t\to\infty$ to the geodesic sphere of radius $\rho_{\infty}$ (without correction of ambient isometry). This completes the proof of Theorem \ref{theo}.


\end{document}